\documentclass[leqno,11pt,twoside]{amsart}

\usepackage{geometry}
\usepackage{physics}
\usepackage{times}

\usepackage[all]{xy} 

\usepackage{amsmath, amssymb, amsfonts, latexsym, mdwlist, amsthm,amscd}
\usepackage{caption, subcaption}
\usepackage{tikz}
\usetikzlibrary{calc,trees,positioning,arrows,chains,shapes.geometric,%
	decorations.pathreplacing,decorations.pathmorphing,shapes,%
	matrix,shapes.symbols}

\tikzset{
	>=stealth',
	punktchain/.style={
		rectangle,
		rounded corners,
		draw=black, thick,
		minimum height=3em,
		text centered,
		on chain},
	line/.style={draw, thick, <-},
	element/.style={
		tape,
		top color=white,
		bottom color=blue!50!black!60!,
		minimum width=8em,
		draw=blue!40!black!90, very thick,
		text width=10em,
		minimum height=3.5em,
		text centered,
		on chain},
	every join/.style={->, thick,shorten >=1pt},
	decoration={brace},
	tuborg/.style={decorate},
	tubnode/.style={midway, right=2pt},
}

\usepackage{cite}
\usepackage{url}

\usepackage{verbatim}
\usepackage{mathtools}

\usepackage{pgf,tikz}
\usepackage{tikz-cd}
\usetikzlibrary{calc,matrix,patterns,shadings,backgrounds,fit}
\pgfdeclarelayer{myback}
\pgfsetlayers{myback,background,main}
\usepackage{stmaryrd}
\usepackage{diagbox}
\usepackage{extarrows}

\usepackage[bookmarks, colorlinks, breaklinks, pdftitle={quintic},
pdfauthor={ Chunyi Li}]{hyperref}
\hypersetup{linkcolor=blue,citecolor=blue,filecolor=black,urlcolor=blue}

\usepackage{paralist}
\setdefaultenum{(a)}{(i)}{}{}
\usepackage{enumitem} 

\def\C{\ensuremath{\mathbb{C}}}

\def\H{\ensuremath{\mathbb{H}}}

\def\R{\ensuremath{\mathbb{R}}}
\def\Z{\ensuremath{\mathbb{Z}}}

\def\ch{\mathop{\mathrm{ch}}\nolimits}

\def\Coh{\mathop{\mathrm{Coh}}\nolimits}

\def\deg{\mathop{\mathrm{deg}}}

\def\ext{\mathop{\mathrm{ext}}\nolimits}
\def\Ext{\mathop{\mathrm{Ext}}\nolimits}

\def\HH{\mathrm{H}}

\def\Hom{\mathop{\mathrm{Hom}}\nolimits}

\def\min{\mathop{\mathrm{min}}\nolimits}

\def\rk{\mathop{\mathrm{rk}}}

\DeclarePairedDelimiter\floor{\lfloor}{\rfloor}

\def\abs#1{\left\lvert#1\right\rvert}

\makeatletter
\newtheorem*{rep@theorem}{\rep@title}
\newcommand{\newreptheorem}[2]{%
\newenvironment{rep#1}[1]{%
 \def\rep@title{#2 \ref{##1}}%
 \begin{rep@theorem}}%
 {\end{rep@theorem}}}
\makeatother

\newtheorem{Thm}{Theorem}[section]
\newreptheorem{Thm}{Theorem}
\newtheorem{Prop}[Thm]{Proposition}

\newtheorem{Lem}[Thm]{Lemma}

\newtheorem{Cor}[Thm]{Corollary}
\newreptheorem{Cor}{Corollary}
\newtheorem{Con}[Thm]{Conjecture}

\newtheorem{Obs}[Thm]{Observation}

\newtheorem{thm-int}{Theorem}

\theoremstyle{definition}
\newtheorem{Def-s}[Thm]{Definition}
\newtheorem{Def}[Thm]{Definition}
\newtheorem{Rem}[Thm]{Remark}

\def\C{\ensuremath{\mathbb{C}}}

\def\H{\ensuremath{\mathbb{H}}}

\def\R{\ensuremath{\mathbb{R}}}
\def\Z{\ensuremath{\mathbb{Z}}}

\def\cA{\ensuremath{\mathcal A}}

\def\cF{\ensuremath{\mathcal F}}

\def\cO{\ensuremath{\mathcal O}}

\def\cT{\ensuremath{\mathcal T}}

\usepackage{todonotes}

\def\nabH{\nu_{\alpha, \beta,H}}

\def\db{\bar{\Delta}}
\def\vb{\bar{v}_H}

\begin{document}

\title[On stability conditions for the quintic threefold]{On stability conditions for the quintic threefold}

\author{Chunyi Li}
\address{C. L.:\\
Mathematics Institute, University of Warwick,
Coventry, CV4 7AL,
United Kingdom}
\email{C.Li.25@warwick.ac.uk}
\urladdr{https://sites.google.com/site/chunyili0401/}


\begin{abstract} 
We study the Clifford type inequality for a particular type of curves $C_{2,2,5}$, which are contained in smooth quintic threefolds. This allows us to prove some stronger Bogomolov-Gieseker type inequalities for Chern characters of stable sheaves and tilt-stable objects on smooth quintic threefolds. Employing the previous framework by Bayer, Bertram, Macr\`\i, Stellari and Toda, we construct an open subset of stability conditions on every smooth quintic threefold in $\mathbf{P}^4_{\C}$. 
\end{abstract}

\date{\today}

\maketitle

\setcounter{tocdepth}{1}
\tableofcontents
\def\mhk{M^h_k} 
\def\lrd{V^r_d(\abs{H})}
\def\ext{\mathrm{ext}}
\def\Ext{\mathrm{Ext}}
\def\clf{\mathrm{Cliff}}
\def\rg{\frac{r}{g}}
\def\gr{\frac{g}{r}}
\def\fgr{\left\lfloor\frac{g}{r}\right\rfloor}
\def\lf{\left\lfloor}
\def\rf{\right\rfloor}
\def\pp{\mathbf{P}^2}
\def\vvv{\tilde{\mathbf{v}}}
\def\cer{C_{2,2,5}}
\def\ser{S_{2,5}}

\section{Introduction}
The notion of stability conditions on a triangulated category is introduced by Bridgeland in \cite{Bridgeland:Stab}. The existence of stability conditions on three-dimensional projective varieties, and more specifically
on Calabi-Yau threefolds, is often considered as one of the biggest open problem in the theory of Bridgeland
stability conditions in recent years. In series work of \cite{BMT:3folds-BG,BBMT:Fujita,BMS:stabCY3s}, the authors propose a general approach towards the constructions of geometric stability conditions on a smooth projective threefold. The construction involves the notion of tilt-stability for two-term complexes, and the existence of geometric stability conditions relies on a conjectural Bogomolov-Gieseker type inequality for the third Chern character of tilt-stable objects.

Stability conditions are only known to exist on few families of smooth projective threefolds: Fano threefolds \cite{Macri:P3,Benjamin:quadric,Li:Fano,Dulip:Fano,MMSZ:Fano}, Abelian threefolds \cite{Dulip-Antony:I,Dulip-Antony:II,BMS:stabCY3s} and Kummer type threefolds \cite{BMS:stabCY3s}. The smooth quintic threefolds will be the first example of strict Calabi-Yau threefolds that has geometric stability conditions. One need to be cautious that the original conjectural Bogomolov-Gieseker type inequality in \cite{BMT:3folds-BG} does not hold for all threefolds, counterexamples for the blowup at a point of another threefold has been constructed in \cite{Benjamin:conterBG,Benjamin-Christian:conterblowup}. However, due to the flexibility of the construction in \cite{BMT:3folds-BG} as well as the work \cite{Dulip-Toda:bgtostab}, modified Bogomolov-Gieseker type inequality will still imply the existence of stability conditions on such threefolds.

In this paper, we prove the following Bogomolov-Gieseker type inequalities for the second Chern character of slope stable sheaves on smooth quintic threefolds:
\begin{Thm}[Theorem \ref{thm:BGforslopestableX5}]
Let $F$ be a torsion free $\mu_H$-slope semistable sheaf on a smooth quintic threefold $(X,H)$. Suppose $\frac{H^2\ch_1(F)}{H^3\rk(F)}\in[-1,1]$, then
\begin{align*}
H\ch_2(F)\leq \begin{cases}
-\frac{1}2\abs{H^2\ch_1(F)}, & \text{when } \abs{\frac{H^2\ch_1(F)}{H^3\rk(F)}} \in[0, \frac{1}4];\\
\frac{1}2\abs{H^2\ch_1(F)} - \frac{5}4 \rk(E), & \text{when } \abs{\frac{H^2\ch_1(F)}{H^3\rk(F)}} \in [\frac{1}4, \frac{3}4];\\
\frac{3}2\abs{H^2\ch_1(F)} - 5\rk (E) , & \text{when } \abs{\frac{H^2\ch_1(F)}{H^3\rk(F)}} \in [\frac{3}4, 1].
\end{cases}
\end{align*}
The `$=$' can hold only when $\abs{\frac{H^2\ch_1(F)}{H^3\rk(F)}}\in\frac{1}4\Z$. Moreover, \text{when } $\abs{\frac{H^2\ch_1(F)}{H^3\rk(F)}} \in [0, \frac{1}{10}]\cup[\frac{9}{10},1]$, we have the following stronger bound:
$H\ch_2(F)\leq\frac{3}2\frac{(H^2\ch_1(F))^2}{H^3\rk(F)}-\abs{H^2\ch_1(F)}$.
\label{thm:bgforslopeonX5intro}
\end{Thm}

In a special case that when $\frac{H^2\ch_1(F)}{H^3\rk(F)}=-\frac{1}2$, we have $\Delta(F)H\geq 1.25\rk(F)^2$, which is a slightly weaker inequality than that in \cite[Conjecture 1.2]{Toda:Gepner}. In particular, it implies the rank $2$ case as that in \cite[Proposition 1.3]{Toda:Gepner}.

Theorem \ref{thm:bgforslopeonX5intro} implies \cite[Conjecture 4.1]{BMS:stabCY3s} for smooth quintic threefolds with a little constrain on the parameters $(\alpha,\beta)$, for which we will review in the next few paragraphs.
\begin{Thm}[Theorem \ref{thm:boginqforx5}]
Conjecture 4.1 in \emph{ \cite{BMS:stabCY3s}} holds for smooth quintic threefolds when the parameters satify $\alpha^2+(\beta-\floor{\beta}-\frac{1}2)^2>\frac{1}4.$
\label{thm:bgforX5intro}
\end{Thm}

Employing the framework in  \cite{BMS:stabCY3s,BMT:3folds-BG,Dulip-Toda:bgtostab}, Theorem \ref{thm:bgforX5intro} allows us to construct a family of Bridgeland stability conditions on the bounded derived category of coherent sheaves on each smooth quintic threefold. To give the accurate statement, we introduce some notions from \cite{BMT:3folds-BG,BBMT:Fujita,BMS:stabCY3s} and briefly summarize the construction of stability conditions on a quintic threefold.

\textbf{Stability conditions on smooth quintic threefolds:} Let $(X,H)$ be a smooth quintic threefold with $H=[\cO_X(1)]$, let $D^b(X)$ be the bounded derived category of coherent sheaves on $X$. As shown in \cite[Proposition 5.3]{Bridgeland:Stab}, a stability condition on $D^b(X)$ is equivalently determined by a pair $\sigma=(Z,\mathcal A)$, where the central charge $Z:K_0(\cA)\rightarrow \C$ is a group homomorphism and $\cA\subset D^b(X)$ is the heart of a bounded t-structure, which have to satisfy the following three properties.
\begin{enumerate}
\item For any non-zero object $E\in \cA$, its central charge $Z([E])\in \R_{>0}\cdot e^{(0,1]\pi i}$. 

This allows us to define a notion of slope-stability on $\cA$ via the slope function $$\nu_\sigma(E)\coloneqq-\frac{\Re Z([E])}{\Im Z([E])}.$$
\item With respect to the slope-stability $\nu_{\sigma}$, each non-zero object $E\in \cA$ admits a unique Harder-Narasimhan filtration:
$$0=E_0\subset E_1\subset\dots \subset E_m=E$$
such that: each quotient $F_i\coloneqq E_i/E_{i-1}$ is $\mu_\sigma$-slope semistable with $\nu_{\sigma}(F_1)>\nu_{\sigma}(F_2)>\dots>\nu_{\sigma}(F_m)$. We set $\nu_{\sigma}^+(E)\coloneqq \nu_{\sigma}(F_1)$ and $\nu_{\sigma}^-(E)\coloneqq \nu_{\sigma}(F_m)$.
\item (support property) There is a constant $C>0$ such that for all  semistable object $E\in \cA$, we have $\|[E]\|\leq C\abs{Z([E])}$, where $\|\cdot\|$ is a fixed norm on $K_0(X)\otimes \R$.
\end{enumerate}
Under the framework of \cite{BMT:3folds-BG,BBMT:Fujita,BMS:stabCY3s}, the heart $\cA$ of the stability condition is constructed by `double-tilting' $\Coh(X)$. Denote $\mu_H$ as the slope stability on $\Coh(X)$. For any object $E\in\Coh(X)$, let $\mu^+_H(E)$ ($\mu^-_H(E)$) be the maximum (minimum) slope of its Harder-Narasimhan factors. The first tilting-heart $\Coh^{ \beta,H}(X)\subset D^b(X)$ with parameter $\beta\in\R$ is the extension-closure $\langle \mathcal T_{\beta,H},\mathcal F_{H,\beta}[1]\rangle$, where
\begin{align*}
\mathcal T_{\beta,H}=\{E\in\Coh(X)|\mu^-_H(E)>\beta\};\;\;\;
\mathcal F_{\beta,H}=\{E\in\Coh(X)|\mu^+_H(E)\leq \beta\}.
\end{align*}

Given $\alpha\in \R_{>0}$, we may define the tilt-slope function for objects in  $\Coh^{\beta,H}(X)$ as follows: for an object $E\in \Coh^{\beta,H}(X)$, its tilt-slope function is defined as
\begin{align}
\nu'_{\alpha,\beta,H}(E)\coloneqq\begin{cases}
\frac{H\ch_2^{\beta H}(E)-\frac{\alpha^2}2 H^3\ch_0(E)}{H^{2}\ch^{\beta H}_1(E)}, & \text{ when } H^{2}\ch^{\beta H}_1(E)>0;\\
+\infty, & \text{ when } H^{2}\ch^{\beta H}_1(E)=0.
\end{cases}\label{eq:nuabH'}
\end{align}

The explicit formulas of twisted Chern characters $\ch_i^{\beta H}$ are given at the beginning of Section \ref{sec:2}.

The heart $\cA^{\alpha,\beta,H}(X)\subset D^b(X)$ is defined  as $\langle \cT_{\alpha,\beta,H}',\cF_{\alpha,\beta,H}' [1]\rangle,$
where
\begin{align*}
\cT_{\alpha,\beta,H}'\coloneqq\{E\in \Coh^{\beta,H}(X)|\nabH'^-(E) >0\};\\
\cF_{\alpha,\beta,H}'\coloneqq\{E\in \Coh^{\beta,H}(X)|\nabH'^+(E)\leq 0\}.
\end{align*}
The central charge on $\cA^{\alpha,\beta,H}(X)$ is defined as that in \cite[Lemma 8.3]{BMS:stabCY3s}:
\begin{align}
    Z_{\alpha,\beta,H}^{a,b}\coloneqq (-\ch_3^{\beta H}+bH\ch_2^{\beta H}+aH^2\ch_1^{\beta H})+i(H\ch_2^{\beta H}-\frac{\alpha^2}2H^3\ch_0). \label{eq:centralcharge}
\end{align}

As a corollary of \cite[Conjecture 4.1]{BMS:stabCY3s} employing the framework in \cite{BMT:3folds-BG,BBMT:Fujita,BMS:stabCY3s}, the construction above offers us a family of stability conditions.
\begin{Thm}[Theorem \ref{thm:bgforX5intro}, {\cite[Theorem 8.6, Proposition 8.10]{BMS:stabCY3s}}]
There is a continuous family of Bridgeland stability conditions $\sigma_{\alpha,\beta,H}^{a,b}=(Z_{\alpha,\beta,H}^{a,b}(X),\cA^{\alpha,\beta,H}(X))$ on each smooth quintic threefold $(X,H)$, parameterized by the set $(\alpha,\beta,a,b)\in \R_{>0}\times \R\times \R_{>0}\times \R$ such that 
$$\alpha^2+(\beta-\floor{\beta}-\frac{1}2)^2>\frac{1}4; \text{ and }a>\frac{\alpha^2}6+\frac{1}2\abs{b}\alpha.$$
This family is a slice of the $\widetilde{GL}^+_2(\R)$-action on an open subset of the space of stability conditions on $D^b(X)$.
\label{thm:familyofstabonX}
\end{Thm}

The mirror family of $X$ is parameterized by the stack $\mathcal M_K$, which is called the stringy K\"{a}hler moduli space of $X$:  $$\mathcal M_K:=\left[\{\psi\in \mathbb C|\psi^5\neq 1\}/\mu_5\right].$$
Here the generator of $\mu_5$ acts on $\mathbb C$ by the multiplication of $e^{\frac{2\pi i}5}$. Based on the papers \cite{Aspinwall:Dbranes-CY}, \cite[Remark 3.9]{Bridgeland:embeding} and \cite{Toda:Gepner}, it is expected that there is an embedding from the stringy K\"{a}hler moduli space to the double quotient:
$$I:\mathcal M_K \hookrightarrow [\mathrm{Aut}(D^b(X))\setminus \mathrm{Stab}(X)/\mathbb C].$$
We refer readers to \cite[Section 7.1]{Aspinwall:Dbranes-CY} and \cite[Section 3]{Toda:Gepner} for more detailed discussions and predictions on the formula of centrals charge and heart structures. Under this embedding, the images of (the neighbourhoods of) three special points are of particular interests:
\begin{itemize}
    \item the \textit{large volume(radius) limit} at the point $\psi=\infty$;
    \item the \textit{conifold gap point} at the point $\psi^5=1$;
    \item the \textit{Gepner point} the point $\psi=0$.
\end{itemize}
Up to the actions by Aut$(D^b(X))$ and $\mathbb C$, the images of the neighborhood of the large volume limit are expected to be expressed by geometric stability conditions with predicted central charge:

$$Z_{\beta,t,H}\coloneqq (-\ch_3^{\beta H} +\frac{t^2}2H^2\ch_1^{\beta H})+i(tH\ch_2^{\beta H}-\frac{5t^3}6H^3\ch_0),$$

where $\beta\in \R$ and $t>0$. Scaling the imaginary part of (\ref{eq:centralcharge}) by $t$, let $b=0$, $a=\frac{t^2}2$ and $\alpha=\frac{\sqrt{15}}3t$; we get all such central charges for $t>1$. In particular, the space of stability conditions constructed in Theorem \ref{thm:familyofstabonX} contains a neighbourhood of the large volume limit. 

Up to the actions by Aut$(D^b(X))$, the limit of central charges near the conifold gap point is expected to satisfy $Z(\mathcal O_X)=0$. Note that in (\ref{eq:centralcharge}), one may let $\beta=0$ and $\alpha\rightarrow 0$ so that the kernel of the central charge will tend to the character of $\mathcal O_X$. In particular, the space of stability conditions constructed in Theorem \ref{thm:familyofstabonX} contains parts of the neighbourhood of conifold gap point.

The image of the Gepner point is expected to be represented by a stability condition that is fixed by the action $(\mathrm{ST}_{\mathcal O_X}\circ \otimes \mathcal O(H),-\frac{2}5)$. By \cite{Toda:Gepner}, to construct such a stability condition, we prior need a Bogomolov-Gieseker type inequality which is `slighly' stronger than that in Theorem \ref{thm:bgforslopeonX5intro}. We hope to prove this better bound in some future projects after introducing more techniques in the paper \cite{Li:bgforfanohigh}.

\subsection{Organisation and Approach:}
The logic flow of the proof is as follows:
\begin{center}
\textcircled{4} Proposition \ref{prop:clfboundsforC225} $\xLongrightarrow[\text{\textcircled{3}}]{ }$ Theorem \ref{thm:BGforslopestableX5} $\xLongrightarrow[\text{\textcircled{2}}]{}$ Proposition \ref{prop:BGforbnstab} $\xLongrightarrow[\text{\textcircled{1}}]{\text{Theorem \ref{thm:bnimpliesall}}}$ Theorem \ref{thm:boginqforx5}.
\end{center}
Each statement above is an inequality for characters of certain semistable objects. Each `$\implies$' only relies on the previous inequality but not relates to the arguments for that inequality.   
The argument in \textcircled{1} follows the technique in \cite[Section 5]{BMS:stabCY3s}, it is also originated from the idea in \cite[Section 2.2]{Macri:P3}. Naively speaking, by \textcircled{1} we may reduce the inequality for stable objects with respect to every tilt-slope functions to a single type: the so-called `Brill-Noether' stable objects.  The mainstream of the argument in \textcircled{2} is to follow the technique developed for Fano threefolds as that in \cite{Li:Fano,Dulip:Fano,MMSZ:Fano}. However, in the Calabi-Yau threefold case, we don't have some of the $\Hom$ vanishings as that in the Fano threefolds case. Instead, we need to estimate the $\hom(\cO_X,E)$ for Brill-Noether stable objects. The original idea for this estimation via stability conditions, as far as the author knows, first appears in \cite{wallcrossing-BrillNoether} which reproves the Brill-Noether generality of certain curves on K3 surfaces as that in \cite{Lazarsfeld:BN-Petri}. The estimation for $\hom(\cO_X,E)$ necessarily relies on a stronger Bogomolov-Gieseker type inequality for the second Chern character of slope stable objects, which is the statement of Theorem \ref{thm:BGforslopestableX5}. In addition to Proposition \ref{prop:clfboundsforC225}, the argument in \textcircled{3} relies on two techniques: the deformation of stability conditions and Feyzbakhsh's restriction lemma. A similar deformation argument has been used in \cite{Li:Fano} for the case of Fano threefolds with index one. The restriction lemma first appears in \cite{Fe16}, where the author shows the stability of vector bundles on curves restricted from a K3 surface. More details about the restriction technique via stability conditions appear in Feyzbakhsh's thesis. The argument \textcircled{3} can produce more Bogomolov-Gieseker type inequalities for the first two Chern characters for several other varieties. Some results focused on this direction will appear soon in \cite{Li:bgforfanohigh}. \textcircled{4} Proposition \ref{prop:clfboundsforC225} is the Clifford type bound for the dimension of global sections of stable vector bundles on curves $C_{2,2,5}$, the complete intersection of two quadratics and a quintic hypersurface in $\mathbf{P}^4$. As a topic of its own interest, several general results on the Clifford type bound for curves can be found in \cite{Aprodu2014Restricted,Lange2015g4,Lange2015g5,Lange2017g6,Mercat2002CLIFFORD}. It is pity that none of the results mentioned above fit in our situation since we need the sharp bounds at some critical slopes $\mu=5$, $10$, $30$ and $35$. Based on the idea in \cite{Fe17}, together with Feyzbakhsh, we develop our own methods to estimate the Clifford type bound for curves supported on K3 and Fano surfaces via stability conditions in \cite{FL:clford}.  Especially for this case, we think $C_{2,2,5}$ as a curve on a degree four del Pezzo surface.  More introductions about the technical details in \textcircled{4} can be found in \cite{Fe17,FL:clford}.

We organize the paper slightly different from the logic flow. Section \ref{sec:2} is to fix some notations and  to collect  some lemmas and tools that will be useful in every other section. In section \ref{sec:main}, we assume the result in Theorem \ref{thm:BGforslopestableX5} and directly prove our main Theorem \ref{thm:boginqforx5}. We make this arrangement since the arguments in this part are more well-established, also we would like to convince the reader that a stronger Bogomolov-Gieseker type inequality for the second Chern character of slope stable sheaves will imply  Bogomolov-Gieseker type inequality for the third Chern character of tilt-stable complexes at this early stage. Section \ref{sec:clf} is devoted to proving the Clifford type bound for the dimension of global sections of a stable vector bundle on the curve $C_{2,2,5}$. This section involves a certain amount of computations. As for the convenience of the readers, there is no harm to skip these details first. Section \ref{sec:5} is to proof the stronger Bogomolov-Gieseker type inequality for the surfaces $S_{2,5}$ based on the inequality in Proposition \ref{prop:clfboundsforC225}.

\textbf{Acknowledgement:}  The author is a Leverhulme Early Career Fellow at the University of Warwick and would like to acknowledge the Leverhulme Trust for the support. The work was initiated when the author was a postdoc at the University of Edinburgh and was supported by the ERC starting grant `WallXBirGeom' 337039. I would like to thank Arend Bayer, Soheyla Feyzbakhsh, Sheldon Katz, Emanuele Macr\`\i, Laura Pertusi, Benjamin Schmidt, Junliang Shen, Paolo Stellari and Xiaolei Zhao for many useful discussions on this topic. I am grateful to Tom Bridgeland and my advisor Tom Nevins for encouraging me to carry on thinking this problem. The main breakthrough of this project is done during my visit at BICMR in July 2018. I would like to thank Zhiyu Tian, Chenyang Xu and Qizheng Yin for their hospitality.

\section{Background: tilt-stability condition and wall-crossing}
\label{sec:2}
\subsection{Stability condition: notations and conventions}

In this section, we review the notion of stability and tilt-stability for smooth varieties introduced in \cite{Bridgeland:Stab,BMT:3folds-BG,Dulip-Toda:bgtostab}. We then recall the conjectural Bogomolov-Gieseker type inequality for tilt-stable complexes proposed
there.

Let $X$ be a smooth projective complex variety and $H\in NS(X)_{\R}$  be a real ample divisor class. Let the dimension of $X$ be $n$, in this paper, $n$ will always be $2$ or $3$. For an arbitrary divisor class $B\in NS(X)_{\R}$ , we will always denote the twisted Chern characters as follows:
\begin{align*}
& \ch_0^{B}=\ch_0=\rk
& \ch_2^{B}=\ch_2-B\ch_1+\frac{B^2}{2}\ch_0\\
& \ch_1^{B}=\ch_1-B\ch_0
& \ch_3^{B}=\ch_3-B\ch_2+\frac{B^2}{2}H^2\ch_1-\frac{B^3}{6}\ch_0.
\end{align*}
In this paper, we are mainly interested in smooth quintic threefold whose $NS(X)_{\R}$ is of rank $1$, we will always assume $B=\beta H$ for some $\beta\in \R$.  The $\mu_{H}$-slope of  a coherent sheaf $E$ on $X$ is defined as 
\begin{align*}
\mu_H(E)=\begin{cases}
\frac{H^{n-1}\ch_1(E)}{H^n\ch_0(E)}, & \text{when }\ch_0(E)\neq 0;\\
+\infty, & \text{when }\ch_0(E)=0.
\end{cases}
\end{align*}

\begin{Def}
A coherent sheaf $E$ is called slope (semi)stable if for any non-trivial subsheaf $F\hookrightarrow E$, we have $$\mu_H(F)<(\leq) \mu_H(E/F).$$
\label{def:slopestable}
\end{Def}

Each coherent sheaf $E$ admits a unique Harder-Narasimhan filtration:
$$0=E_0\subset E_1\subset\dots \subset E_m=E$$
such that: each quotient $F_i\coloneqq E_i/E_{i-1}$ is slope semistable; and $\mu_H(F_1)>\mu_H(F_2)>\dots>\mu_H(F_m)$. We set $\mu^+_H(E)\coloneqq \mu_H(F_1)$ and $\mu^-_H(E)\coloneqq \mu_H(F_m)$.

There exists \emph{torsion pairs} $(\mathcal T_{\beta,H},\mathcal F_{\beta,H})$ in $\Coh(X)$ defined as follows:
\begin{align*}
\mathcal T_{\beta,H}=\langle \text{semistable } E\in \Coh(E) \text{ with } \mu_H(E)>\beta\rangle=\{E|\mu^-_H(E)>\beta\};\\
\mathcal F_{\beta,H}=\langle \text{semistable } E\in \Coh(E) \text{ with } \mu_H(E)\leq\beta\rangle=\{E|\mu^+_H(E)\leq \beta\}.
\end{align*}

\begin{Def}
We let $\Coh^{\beta,H}(X)\subset D^b(X)$ be the extension-closure $$\langle \mathcal T_{\beta,H},\mathcal F_{\beta,H}[1]\rangle.$$

\label{def:tiltheart}
\end{Def}
By the general theory on tilting heart in \cite{Happel-al:tilting}, $\Coh^{\beta,H}(X)$ is the heart of a t-structure in $D^b(X)$. Given $\alpha\in \R$, we may define the tilt-slope function for objects in  $\Coh^{\beta,H}(X)$ as follows: for an object $E\in \Coh^{\beta,H}(X)$, its tilt-slope function
\begin{align*}
\nu_{\alpha,\beta,H}(E)\coloneqq\begin{cases}
\frac{H^{n-2}\ch_2(E)-\alpha H^n\ch_0(E)}{H^{n-1}\ch^{\beta H}_1(E)}, & \text{ when } H^{n-1}\ch^{\beta H}_1(E)>0;\\
+\infty, & \text{ when } H^{n-1}\ch^{\beta H}_1(E)=0.
\end{cases}
\end{align*}
\begin{Def}
An object $E\in\Coh^{\beta,H}(X)$ is called $\nabH$-tilt slope (semi)stable if for any non-trivial subobject $F\hookrightarrow E$ in $\Coh^{\beta,H}(X)$, we have $$\nabH(F)<(\leq) \nabH(E/F).$$
An object $E\in D^b(X)$ is called $\nabH$-tilt (semi)stable if $E[m]\in \Coh^{\beta,H}(X)$ is $\nabH$-tilt (semi)stable for some homological shift $m\in \Z$.
\label{def:tiltslopestable}
\end{Def}
The tilt slope stability also admits Harder-Narasimhan property when $\alpha>\frac{\beta^2}2$. For an object $E\in\Coh^{H,\beta}(X)$ we may write $\nabH^+(E)$ and $\nabH^-(E)$ for the maximum and minimum slopes of its semistable factors respectively.

We also write the central charge $$Z_{\alpha,\beta,H}(E)\coloneqq -(H^{n-2}\ch_2(E)-\alpha H^n\ch_0(E))+iH^{n-1}\ch^{\beta H}_1(E)$$
for an object $E\in\Coh^{\beta,H}(X)$.
\begin{Rem}
The formula $\nabH$ is re-parameterized from the one in \cite[Section 4]{BMS:stabCY3s}. Let the tilt-slope function in  \cite[Section 4]{BMS:stabCY3s} be $\nabH'$, then 
$$\nabH'=\nu_{\frac{1}2(\alpha^2+\beta^2),\beta,H}-\beta.$$
In particular, an object $E\in\Coh^{\beta,H}(X)$ is $\nabH'$-tilt (semi)stable (in the sense of \cite{BMS:stabCY3s}) if and only if $\nu_{\frac{1}2(\alpha^2+\beta^2),\beta,H}$-tilt (semi)stable. We use $\nabH$ as it is more convenient to compare the slopes of objects via pictures.
\label{rem:tiltslope}
\end{Rem}

\begin{Def}
Let $E$ be an object in $D^b(X)$, we define its $H$-discriminant as
$$\db_H(E)\coloneqq (H^{n-1}\ch_1(E))^2-2H^n\ch_0(E)\cdot H^{n-2}\ch_2(E).$$
\label{def:Hdiscriminant}
\end{Def}

\begin{Thm}[Bogomolov Inequality {\cite{Bogomolov:Ineq}}, {\cite[Theorem 7.3.1]{BMT:3folds-BG}, \cite[Proposition 2.21]{Dulip-Toda:bgtostab}}] 
Let $X$ be a smooth projective variety, and $H\in NS(X)_{\R}$ an ample class. Assume that $E$ is $\nabH$-tilt semistable for some $\alpha>\frac{1}2\beta^2$, then  $\db_H(E)\geq 0$.
\label{thm:discrimbg}
\end{Thm}

The main goal of this paper is on the following conjectural Bogomolov-Gieseker inequality for $\nabH$-tilt semistable objects:
\begin{Con}[{\cite[Conjecture 4.1]{BMS:stabCY3s}}{\cite[Conjecture 2.7] {BMT:3folds-BG}}]
Let $X$ be a smooth projective threefold, and $H\in NS(X)_{\R}$ an ample class. Assume that $E$ is $\nabH$-tilt semistable for some $\alpha>\frac{1}2\beta^2$, then  
\begin{align}
Q_{\alpha,\beta}(E)\coloneqq (2\alpha-\beta^2)\db_H(E)+4(H\ch^{\beta H}_2(E))^2-6H^2\ch^{\beta H}_1(E)\ch^{\beta H}_3(E)\geq 0.
\label{eq:bgtilt}
\end{align}
\label{conj:4.1}
\end{Con}

In this paper, we will prove this conjecture for smooth quintic threefolds with a little assumption on $\alpha$.

\begin{Thm}
Let $X$ be a smooth projective quintic threefold, and $H=[\mathcal O_X(1)]$. Assume that $E$ is $\nabH$-tilt semistable for some $\alpha>\frac{1}2\beta^2+\frac{1}2(\beta-\floor \beta)(\floor \beta+1-\beta)$, then the inequality (\ref{eq:bgtilt}) holds.
\label{thm:boginqforx5}
\end{Thm}

\subsection{Recollection of lemmas} Let $X$ be a smooth projective variety and $H\in NS(X)_{\R}$ be a real ample divisor class. For an object $E\in D^b(X)$, we write 
\begin{align*}
& \vb(E)\coloneqq(H^n\ch_0(E),H^{n-1}\ch_1(E),H^{n-2}\ch_2(E));\\
\text{and } & p_H(E)\coloneqq\left(\frac{H^{n-2}\ch_2(E)}{H^n\ch_0(E)},\frac{H^{n-1}\ch_1(E)}{H^n\ch_0(E)}\right)\text{, when $\vb(E)\neq 0$}.
\end{align*}
Let $\alpha,\beta\in \R$ be the parameters for tilt-slope functions, unless mentioned otherwise, we will always assume $\alpha>\frac{1}2\beta^2$.

\begin{Lem}
Let $E\in \Coh^{\beta_0,H}(X)$ be a $\nu_{\alpha_0,\beta_0,H}$-tilt stable object for some $\alpha_0>\frac{1}2\beta_0^2$, then we have the following properties.
\begin{enumerate}
\item (Openness) There exists an open set of neighborhood $U$ of $(\alpha_0,\beta_0)$ such that for any $(\alpha,\beta)\in U$, the object $E$  is $\nabH$-tilt stable.
\item (Bertram's Nested Wall Theorem) The object $E$ is $\nabH$-tilt stable for any $\{(\alpha,\beta)|\alpha>\frac{1}2\beta^2\}$ on the line through the points $(\alpha_0,\beta_0)$ and $p_H(E)$. More precisely, the object $E$ is $\nabH$-tilt stable for $(\alpha,\beta)$ such that the determinant 
$$\det \begin{pmatrix}
1 & \alpha &\beta\\
1 & \alpha_0 &\beta_0\\
H^n\ch_0(E)  & H^{n-2}\ch_2(E) & H^{n-1}\ch_1(E)
\end{pmatrix}=0.$$
The statement also holds for semistable case.
Moreover, when $X$ is a threefold, 
\begin{align}
H^{n-1}\ch_1^{\beta H}(E)Q_{\alpha_0,\beta_0}(E)=H^{n-1}\ch_1^{\beta_0 H}(E)Q_{\alpha,\beta}(E).\label{eq:Qequalontheline}
\end{align}
\item[(b$'$)] Let $F$ be an object in $\Coh^{\beta_0,H}(X)$ such that $p_H(F)$ is on the line through the points $(\alpha_0,\beta_0)$ and $p_H(E)$, then $\nu_{\alpha_0,\beta_0,H}(E)=\nu_{\alpha_0,\beta_0,H}(F).$ More precisely, the requirements on $E$ and $F$ are as follows: both $\vb(E)$ and $\vb(F)$ are not zero and 
the determinant 
$$\det \begin{pmatrix}
1 & \alpha_0 &\beta_0\\
H^n\ch_0(E)  & H^{n-2}\ch_2(E) & H^{n-1}\ch_1(E)\\
H^n\ch_0(F)  & H^{n-2}\ch_2(F) & H^{n-1}\ch_1(F)
\end{pmatrix}=0.$$
\item (Destabilizing walls) The set $\{(\alpha,\beta)\in \R^2|\alpha>\frac{1}2\beta^2, E$ is strictly $\nabH$-tilt semistable$\}$ is empty or a union of line segments and rays.
\end{enumerate}
\label{lem:nestedandopen}
\end{Lem}
\begin{proof}
The first and third statements are in \cite[Corollary 3.3.3]{BMT:3folds-BG} and also in \cite[Appendix B]{BMS:stabCY3s} with more details. The nested wall theorem is in \cite[Theorem 3.1]{Maciocia:walls} and \cite[Lemma 4.3]{BMS:stabCY3s}. As for the equation (\ref{eq:Qequalontheline}), by formally tensoring $\cO(mH)$ on $E$, we may assume that $H^{n-1}\ch_1(E)=0$. The left hand side then can be simplified as:
$$4\beta H^n\ch_0H^{n-2}\ch_2(\alpha_0H^n\ch_0-H^{n-2}\ch_2)-6\beta\beta_0 (H^n\ch_0)^2H^{n-3}\ch_3.$$
This equals the right hand side since the zero determinant implies: $$\beta (\alpha_0H^n\ch_0-H^{n-2}\ch_2)=\beta_0 (\alpha H^n\ch_0-H^{n-2}\ch_2).$$
Part (b$'$) is a direct computation by the definition of $\nabH$.
\end{proof}

The following lemma from \cite{BMS:stabCY3s} will be very useful in the technique of deforming tilt-stabilities. We list it here for the convenience of readers.

\begin{Lem}[{\cite[Corollary 3.10]{BMS:stabCY3s}}]
Let $E$ be a strictly $\nabH$-tilt semistable object with $\nabH(E)\neq +\infty$. Then for any of its Jordan-H\"older factor $E_i$ of $E$, we have
$$\db_H(E_i)\leq \db_H(E).$$
The equality holds only when $\vb(E_i)$ is proportional to $\vb(E)$ and $\db_H(E)=\db_H(E_i)=0$.  
\label{lem:discriminnatdec}
\end{Lem}

\begin{Def}
We call an object $E$ \emph{Brill-Noether stable} if there exists an open subset $U_{\delta}\coloneqq \{ (\alpha,\beta)|\alpha^2+\beta^2<\delta,\alpha>\frac{1}2\beta^2\}$ for some $\delta>0$ such that $E$ is $\nabH$-tilt stable for every $(\alpha,\beta)\in U_{\delta}$. 

We call an object $E$ \emph{Brill-Noether semistable}  if there exists $\delta>0$ such that $E$ is $\nu_{\alpha,0,H}$-tilt semistable for every $0<\alpha<\delta$.

For an object $E\in\Coh^{0,H}(X)$, we denote its Brill-Noether slope by $$\nu_{BN}(E)\coloneqq \begin{cases}
\frac{H^{n-2}\ch_2(E)}{H^{n-1}\ch_1(E)}, & \text{ when } H^{n-1}\ch_1(E)\neq 0;\\
+\infty, & \text{ when } H^{n-1}\ch_1(E)= 0.
\end{cases}$$
\label{def:bnstab}
\end{Def}

On may think the Brill-Noether stability condition also as the `weak stability condition' on the heart $\Coh^{0,H}(X)$ whose central charge is given by  $Z=-H^{n-2}\ch_2+i H^{n-1}\ch_1$. By Lemma \ref{lem:nestedandopen}, an object $E$ with $H^{n-2}\ch_2(E)\neq 0$ is Brill-Noether stable if and only if it is $\nu_{\alpha,\beta,H}$-tilt stable for some $(\alpha,\beta)$ proportional to $p_H(E)$. The Brill-Noether semistability of $E$ implies that $E$ is $\nu_{\alpha,\beta,H}$-tilt semistable for some $(\alpha,\beta)$ proportional to $p_H(E)$.

\begin{Lem} [{\cite[Lemma 6.5]{wallcrossing-BrillNoether}}]
Assume that $E\in \Coh^{0,H}(X)$ is Brill-Noether stable. If $\nu_{BN}(E)>0$, let $W\subset \Hom(\mathcal O_X,E)$ be a subspace, then the object 
\begin{align*}
\tilde E\coloneqq Cone(\mathcal O_X\otimes W\xrightarrow{can} E)
\end{align*}  is in $\Coh^{0,H}(X)$ and Brill-Noether semistable.

If $\nu_{BN}(E)<0$, let $W'\subset (\Hom(E[-1],\mathcal O_X))^*$ be a subspace, then the object 
\begin{align*}
\tilde E'\coloneqq Cone(E[-1]\xrightarrow{can} \mathcal O_X\otimes W')
\end{align*}
is in $\Coh^{0,H}(X)$ and Brill-Noether semistable.
\label{lem:extension}
\end{Lem}
\begin{proof}
We prove the case when $\nu_{BN}(E)>0$, the other case can be proved in a similar way. Note that $\tilde E$ is the canonical extension $$0\rightarrow E\rightarrow \tilde E\rightarrow \mathcal O_X[1]\otimes W\rightarrow 0$$
in $\Coh^{0,H}(X)$. In the case that $\nu_{BN}(E)=+\infty$, for any $\alpha>0$, both $E$ and $\mathcal O_X[1]\otimes W$ are $\nu_{\alpha,0,H}$-tilt semistable with the same slope $+\infty$. Any extended object from them, especially $\tilde E$, is also $\nu_{\alpha,0,H}$-tilt semistable with  slope $+\infty$. 

We may now assume $H^{n-1}\ch_1(E)\neq 0$, then there exists points $(\alpha,\beta)$  proportional to $p_H(E)$ such that $\alpha>\frac{1}2\beta^2$.  For any such $(\alpha,\beta)$, both $\mathcal O_X[1]$ and $E$ are $\nabH$-tilt stable and $\nabH(E)=\nabH(\mathcal O_X[1])$. Their extension $\tilde{E}$ is $\nabH$-tilt semistable. If $\tilde{E}$ is not $\nu_{\epsilon,0,H}$-tilt stable for sufficiently small $\epsilon>0$, then the destabilizing quotient object $\tilde{E}\twoheadrightarrow Q$ in $\Coh^{0,H}(X)$ would necessarily be $\nabH$-tilt semistable with the same slope as $\nabH(\mathcal O_X[1])=\nabH(E)$. Note that $\mathcal O_X[1]$ is $\nu_{\epsilon,0,H}$-stable with slope $$\nu_{\epsilon,0,H}(\mathcal O_X[1])=+\infty>\nu_{\epsilon,0,H}(\tilde E)>\nu_{\epsilon,0,H}(Q),$$ we have $\Hom(\mathcal O_X[1],Q)=0$. Therefore, we must have $\Hom(E,Q)\neq 0$. 

Since $E$ is $\nabH$-tilt stable and $Q$ is $\nabH$-tilt semistable with the same slope, the object $E$ has to be a subobject of $Q$ in $\Coh^{\beta,H}(X)$. Denote the  kernel of $\tilde{E}\twoheadrightarrow  Q$ by $K$. We then have the short exact sequence 
$$0\rightarrow K \rightarrow \tilde{E}/E\rightarrow Q/E\rightarrow 0$$
in $\Coh^{\beta,H}(X)$. By choosing sufficiently small $\beta>0$,  we have $H^{n-1}\ch_1(Q/E)$, $H^{n-1}\ch_1(K)\geq 0$. Note that $\tilde E/E\simeq \cO_X[1]\otimes W$ and $H^{n-1}\ch_1(\mathcal O_X[1])=0$, we must have $H^{n-1}\ch_1(Q/E)=0$. Since $\nabH(Q/E)=\nabH(E)=\nabH(\mathcal O_X[1])$, we have $\vb(Q/E)=(H^n\ch_0(Q/E),0,0)$ and $\db_H(Q/E)=0$. By \cite[Corollary 3.11(c)]{BMS:stabCY3s}, both $Q/E$ and $K$ must be some direct summands of $\mathcal O_X[1]$. By the definition of $\tilde{E}$, there is no non-zero map from $K$ to $\tilde{E}$. Hence,  $\tilde{E}$ is $\nu_{\epsilon,0,H}$-tilt stable for sufficiently small $\epsilon>0$.
\end{proof}

\section{Proof for the main result}\label{sec:main}
The goal of this section is to prove the inequality in Theorem \ref{thm:boginqforx5} with the assumption of Theorem \ref{thm:BGforslopestableX5}. Following the idea in \cite[Section 5]{BMS:stabCY3s}, we first reduce the inequality for every tilt semistable objects to Brill-Noether stable objects.

\begin{Prop}
Let $X$ be a smooth projective quintic threefold, and $H=[\mathcal O_X(1)]$. Assume that $E\in \Coh^{0,H}(X)$ is Brill-Noether stable and $\nu_{BN}(E)\in [-\frac{1}2,\frac{1}2]$, then 
$$Q_{0,0}(E)\coloneqq 4(H\ch_2(E))^2-6H^2\ch_1(E)\ch_3(E)\geq 0.$$
\label{prop:BGforbnstab}
\end{Prop}

\begin{Thm}[{\cite[Theorem 5.4]{BMS:stabCY3s}}]
Proposition \ref{prop:BGforbnstab} implies Theorem \ref{thm:boginqforx5}.
\label{thm:bnimpliesall}
\end{Thm}
\begin{proof}
Suppose Theorem \ref{thm:boginqforx5} does not hold, then by Theorem \ref{thm:discrimbg}, there exists a $\nabH$-tilt semistable object $E\in \Coh^{\beta,H}(X)$ violating inequality (\ref{eq:bgtilt}) with the minimum $\db_H$. Note that the minimum $\db_H$ is by considering all $(\alpha,\beta)$ such that $\alpha>\frac{1}2\beta^2+\frac{1}2(\beta-\floor \beta)(\floor \beta+1-\beta)$ and every $\nabH$-tilt semistable $E$ such that $Q_{\alpha,\beta}(E)<0$. By \cite[Lemma 5.6]{BMS:stabCY3s}, we may assume $\db_H(E)>0$. We may also assume $H^2\ch_1^{\beta H}(E)>0$, since otherwise $H^2\ch_1^{\beta H}(E)=0$ and the inequality (\ref{eq:bgtilt}) holds automatically.

\begin{figure}
\begin{center}
\scalebox{0.8}{

\begin{tikzpicture}[x=30pt,y=30pt]
\tikzset{%
    add/.style args={#1 and #2}{
        to path={%
 ($(\tikztostart)!-#1!(\tikztotarget)$)--($(\tikztotarget)!-#2!(\tikztostart)$)%
  \tikztonodes},add/.default={.2 and .2}}
}

\newcommand\XA{1}

\draw[->] (-2,0) -- (5.5,0)node[above left] {$\frac{H\ch_2}{H^3\ch_0}$};

\draw[->] (0,-3.5)-- (0,0) node [right] {O} --  (0,3.5) node[above left] {$\frac{H^2\ch_1}{H^3\ch_0}$};

\draw[rotate=-90] (0,0) parabola (3.2,5.12);
\draw[rotate=-90] (0,0) parabola (-3.2,5.12);

\coordinate (O) at (0,0);
\coordinate(O1) at (0.5*\XA,1);
\coordinate(O2) at (2,2);
\coordinate(O-1) at (0.5*\XA,-1);
\coordinate(O-2) at (2,-2);
\coordinate(O-3) at (4.5,-3);
\coordinate(O3) at (4.5,3);

\coordinate(E) at (2.1,2.4);
\coordinate(AB) at (2.4,-1);
\coordinate(A0) at(2.23,1);

\node at (O1) {$\bullet$};
\node at (O) {$\bullet$};
\node at (O2) {$\bullet$};
\node at (O-1) {$\bullet$};
\node at (O-2) {$\bullet$};
\node at (O-3) {$\bullet$};
\node at (O3) {$\bullet$};
\node at (E) {$\bullet$};
\node at (AB) {$\bullet$};
\node at (A0) {$\bullet$};
\node [right] at (A0) {$\beta_0=1$};
\node[right] at (O1) {$(0.5,1)$};
\node[above] at (E) {$p_H(E)$};
\node[right] at (AB) {$(\alpha,\beta)$};
\node at (2.68,-1.8) {$W$};
\draw[dashed] (O3)--(O2)--(O1)--(O)--(O-1)--(O-2)--(O-3);

\draw [add= -1 and 1.5] (O1) to (O2) node[above]{$\nu_{BN}(-\otimes \cO(-H))=\frac{1}2$};
\draw [add= -1 and 2.5] (O) to (O1) node[above]{$\nu_{BN}=\frac{1}2$};
\draw [add= -1 and 3.5] (O-1) to (O) node[above left]{$\nu_{BN}(-\otimes \cO(H))=\frac{1}2$};
\draw [add= -1 and 1.5] (O-2) to (O-1) node[left]{$\nu_{BN}(-\otimes \cO(2H))=\frac{1}2$};
\draw [add= -1 and 1.5] (O-3) to (O-2) node[left]{$\nu_{BN}(-\otimes \cO(3H))=\frac{1}2$};
\draw [add= -0.1 and 0.37] (E) to (AB);
\draw [add= 0.3 and -0.9,dashed] (E) to (AB);
\draw [add= -1.37 and 0.8,dashed] (E) to (AB);

\end{tikzpicture}
}
\end{center}
\caption{The condition $\alpha>\frac{1}2\beta^2+\frac{1}2(\beta-\floor \beta)(\floor \beta+1-\beta)$ is equivalent to say that the point $(\alpha,\beta)$ is to the right of the dashed lines.} \label{fig:nuBNofE}
\end{figure}
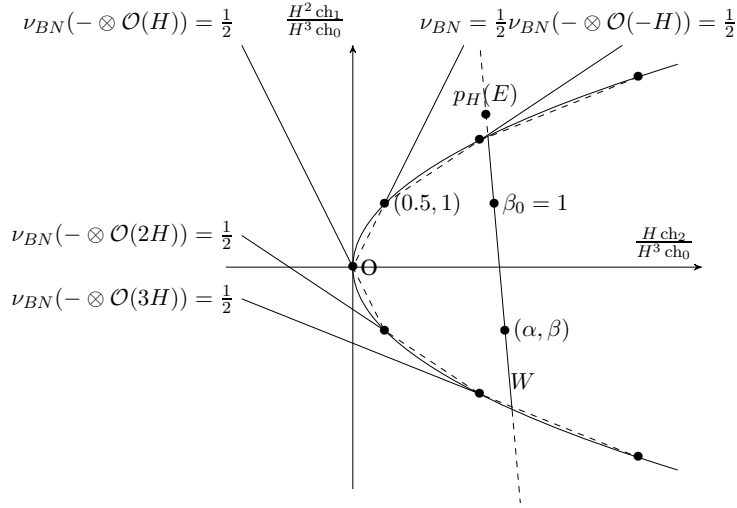

Consider the wall $W$ through $(\alpha,\beta)$ and $p_H(E)$:
$$W\coloneqq \{(\alpha',\beta')|\alpha'>\frac{1}2\beta'^2;(\alpha',\beta'), (\alpha,\beta)\text{ and } p_H(E)\text{ are collinear}\}.$$

 For any $(\alpha',\beta')$ in $W$, by Lemma \ref{lem:nestedandopen}, the object $E$ is $\nu_{\alpha',\beta',H}$-tilt semistable. By Lemma \ref{lem:nestedandopen} part (b), we have $Q_{\alpha',\beta'}(E)<0$.
By the assumption that $\alpha>\frac{1}2\beta^2+\frac{1}2(\beta-\floor \beta)(\floor \beta+1-\beta)$, the wall $W$ contains at least one $(\alpha_0,\beta_0)$ such that $\beta_0$ is an integer.

Moreover, we can choose the integer $\beta_0$ such that 
$$\frac{H\ch^{\beta_0H}_2(E)}{H^2\ch^{\beta_0H}_1(E)}\in[-\frac{1}2,\frac{1}2].$$

Here the integer $\beta_0$ can be determined by the position of $p_H(E)$ as in Figure \ref{fig:nuBNofE}. Or more precisely, the integer $\beta_0$ is in 
$$\begin{cases}
\frac{H^2\ch_1(E)}{H^3\ch_0(E)}-\sqrt{\frac{\db_H(E)}{(H^3\ch_0(E))^2}+\frac{1}4}+\left[-\frac{1}2,\frac{1}2\right]&  \text{,when $\ch_0(E)>0$;} \\
\frac{H^2\ch_1(E)}{H^3\ch_0(E)}+\sqrt{\frac{\db_H(E)}{(H^3\ch_0(E))^2}+\frac{1}4}+\left[-\frac{1}2,\frac{1}2\right]&  \text{,when $\ch_0(E)<0$;} \\
-\frac{H\ch_2(E)}{H^2\ch_1(E)}+\left[-\frac{1}2,\frac{1}2\right]&  \text{,when $\ch_0(E)=0$.} 
\end{cases} $$

By reseting $E=E(-\beta_0H)$, we may assume $\beta_0=0$. In particular,  we may assume that $E$ is $\nu_{\alpha,0,H}$-tilt semistable and $Q_{\alpha,0}(E)<0$. Suppose $E$ becomes strictly $\nu_{\alpha_0,0,H}$-tilt semistable for some $0<\alpha_0\leq\alpha$, then by Lemma \ref{lem:discriminnatdec} and the assumption that $\db_H(E)>0$, for each Jordan-H\"older factor $E_i$, we have $\db_H(E_i)<\db_H(E)$. Note that $Q_{\alpha_0,0}(E)\leq Q_{\alpha,0}(E)<0$. By \cite[Lemma A.6]{BMS:stabCY3s}, there exists a Jordan-H\"older factor $E_i$ such that $Q_{\alpha_0,0}(E_i)<0$. This violates the minimum assumption on $\db_H(E)$. 

Let $(\alpha_1,\beta_1)$ be a point on the wall through $p_H(E)$ and $(0,0)$ when $H\ch_2(E)\neq 0$. By Lemma \ref{lem:nestedandopen},  we have $Q_{\alpha_1,\beta_1}(E)<0$ as $Q_{0,0}(E)\leq Q_{\alpha,\beta}(E)<0$. If $E$ is strictly $\nu_{\alpha_1,\beta_1,H}$-tilt semistable, then any of its Jordan-H\"older factor $E_i$ is Brill-Noether stable and has $\nu_{BN}(E_i)=\nu_{BN}(E)\in[-\frac{1}2,\frac{1}2]$. By Proposition \ref{prop:BGforbnstab}, each factor satisfies $Q_{0,0}(E_i)\geq 0$. By \cite[Lemma A.6]{BMS:stabCY3s} again,  we have $Q_{0,0}(E)\geq 0$.

If $E$ is $\nu_{\alpha_1,\beta_1,H}$-tilt stable or $H\ch_2(E)=0$, then $E$ is Brill-Noether stable. By Proposition \ref{prop:BGforbnstab}, we also have $Q_{0,0}(E)\geq 0$. 

In either case, we get $0>Q_{\alpha,0}(E)\geq Q_{0,0}(E)\geq 0$, which is a contradiction. Therefore, under the assumption of Proposition \ref{prop:BGforbnstab}, Theorem \ref{thm:boginqforx5} holds.
\end{proof}

We now show that Proposition \ref{prop:BGforbnstab} can be implied by the  stronger Bogomolov-Gieseker type inequality for the second Chern character of Brill-Noether stable objects.

\begin{Prop}
Theorem \ref{thm:BGforslopestableX5} implies Proposition \ref{prop:BGforbnstab}.
\label{prop:bg2impliesbg3}
\end{Prop}
\begin{proof}
Let $E\in\Coh^{0,H}(X)$ be a Brill-Noether stable object, we first discuss the case when $\nu_{BN}(E)\in (0,\frac{1}2]$.
There exists $(\alpha,\beta)$ such that $\alpha>\frac{1}2\beta^2$, $0<\frac{\alpha}{\beta}<\nu_{BN}(E)$ and $E$ is $\nabH$-tilt stable. Note that $$\nabH(\mathcal O_X[1])= \frac{\alpha}{\beta}<\nabH(E).$$

\begin{figure}[hbt!]
\begin{center}
\scalebox{0.7}{

\begin{tikzpicture}[scale=1.4]
\tikzset{%
    add/.style args={#1 and #2}{
        to path={%
 ($(\tikztostart)!-#1!(\tikztotarget)$)--($(\tikztotarget)!-#2!(\tikztostart)$)%
  \tikztonodes},add/.default={.2 and .2}}
}

\newcommand\XA{1.5}

\draw[->] (-2.1*\XA,0) -- (2.5*\XA,0) node[above right] {$\frac{H\ch_2}{H^3\ch_0}$};

\draw[->] (0,-2.25)-- (0,0) node [right] {O} --  (0,2.5) node[above left] {$\frac{H^2\ch_1}{H^3\ch_0}$};

\draw[rotate=-90] (0,0) parabola (2,2*\XA);
\draw[rotate=-90] (0,0) parabola (-2,2*\XA);

\coordinate (O) at (0,0);
\coordinate(O1) at (0.5*\XA,1);
\coordinate(O-1) at (0.5*\XA,-1);
\coordinate (AB) at (0.12*\XA,0.4);
\coordinate (E) at (0.75*\XA,2);
\node at (E) {$\bullet$};
\node at (AB) {$\bullet$};
\node at (O1) {$\bullet$};
\node at (O-1) {$\bullet$};
\node[above] at (E) {$p_H(E)$};
\node[right] at (O1) {$p_H(O(H))$};
\node[right] at (O-1) {$p_H(O(-H))$};
\coordinate (B) at (-1.5,1);
\node[left] at (B) {$(\alpha,\beta)$};

\draw[->] (B) --(0.08*\XA,0.43);

\draw [add= 1.3 and 0.2] (O) to (E);
\draw [add= -1.3 and 0.5] (E) to (O) coordinate (E1);
\draw [add= 2.5 and 1.8,dashed] (O) to (O1) node[above] {$\nu_{BN}=\frac{1}2$};
\draw [add= 2.5 and 1.8,dashed] (O) to (O-1) node[right] {$\nu_{BN}=-\frac{1}2$};

\node at (E1) {$\bullet$};
\node[left] at (E1) {$p_H(\tilde E)$};


\draw[red] (0.5*\XA,1)--(0.125*\XA,0.75)--(-0.125*\XA,0.25)--(0,0)--(-0.125*\XA,-0.25)--(0.125*\XA,-0.75)--(0.5*\XA,-1);

\draw[red] (0.5*\XA,1)--(0.625*\XA,1.25)--(1.375*\XA,1.75)--(2*\XA,2);
\draw[red] (0.5*\XA,-1)--(0.625*\XA,-1.25)--(1.375*\XA,-1.75)--(2*\XA,-2);
\end{tikzpicture}
}
\end{center}
\caption{The point $(\alpha,\beta)$ is slightly to the left of the line through $O$ and $p_H(E)$.} \label{fig:phe}
\end{figure}
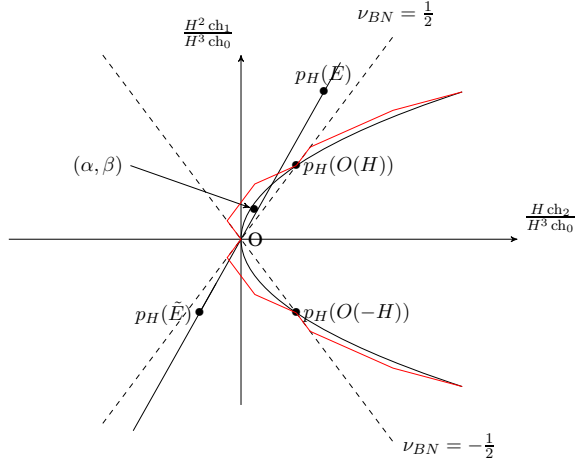

Since both $\mathcal O_X[1]$ and $E$ are $\nabH$-tilt stable, by Serre duality, we have $$\Hom(\mathcal O_X,E[2+i])\simeq (\Hom(E,\mathcal O_X[1-i]))^*=0,$$
for any $i\geq 0$. Consider the object $\tilde E\coloneqq$Cone$(\mathcal O_X\otimes \Hom(\mathcal O_X,E)\rightarrow E)$, by Lemma \ref{lem:extension}, $\tilde E$ is Brill-Noether semistable in $\Coh^{0,H}(X)$. By Theorem \ref{thm:BGforslopestableX5}, the slope $\frac{H^2\ch_1(\tilde E)}{H^3\rk(\tilde E)}$ cannot be in $(-\frac{1}4,0]$. Moreover, either 
\begin{align}
\label{eq:33}\frac{H^2\ch_1(E)}{H^3(\rk(E)-\hom(\mathcal O_X,E))}=\frac{H^2\ch_1(\tilde E)}{H^3\rk(\tilde E)}\notin [-\frac{1}2,-\frac{1}4];
\end{align}
or  by Theorem \ref{thm:BGforslopestableX5}, 
\begin{align}
\frac{H^2\ch_1(\tilde E)}{H^3\rk(\tilde E)}\in [-\frac{1}2,-\frac{1}4] \text { and } H\ch_2(\tilde E)\geq -\frac{1}2H^2\ch_1(\tilde E)-\frac{1}4H^3\rk(\tilde E).
\label{eq:331}
\end{align}
When (\ref{eq:33}) happens, we have 
\begin{align}
\hom(\mathcal O_X,E)<\rk(E)+\frac{2}5H^2\ch_1(E).
\end{align}
When (\ref{eq:331}) happens, we have 
\begin{align}
\hom(\mathcal O_X,E)\leq \rk(E)+\frac{2}5H^2\ch_1(E)+\frac{4}5H\ch_2(E).\label{eq:333}
\end{align}
Note that $H^2\ch_1(E)>0$ and we have assumed that $\nu_{BN}(E)>0$, hence $H\ch_2(E)>0$ and inequality (\ref{eq:333}) always holds. 

Since $\Hom(\mathcal O_X[1], E[i])=0$ for $i\leq -1$, we have 
$$\chi(\mathcal O_X,E)\leq \hom(\mathcal O_X,E).$$
Substitute this to (\ref{eq:333}), recall that $\mathrm{td}_1(X)=0$, $\mathrm{td}_3(X)=\chi(\cO_X)=0$ and $\mathrm{td}_2(X)=\frac{5}6H^2$ (as $\chi(\cO_X(H))=5$),  by Hirzebruch-Riemann-Roch, we have 
$$\ch_3(E)+\frac{5}6 H^2\ch_1(E)=\chi(E)\leq \rk(E)+\frac{2}5H^2\ch_1(E)+\frac{4}5H\ch_2(E).$$
By multiplying $6H^2\ch_1(E)$ and cancelling out some terms on both sides, we have:
\begin{align}
Q_{0,0}(E)  \geq &\frac{13}{5}(H^2\ch_1(E))^2+4(H\ch_2(E))^2-6\rk(E)H^2\ch_1(E)-\frac{24}5H\ch_2(E)H^2\ch_1(E)\nonumber\\
=& \frac{6}5H^2\ch_1(H^2\ch_1-H^3\rk)+\frac{1}5(7H^2\ch_1-10H\ch_2)(H^2\ch_1-2H\ch_2)\label{eq:336}\\=& \frac{6}5H^2\ch_1(2H^2\ch_1-4H\ch_2-H^3\rk)+4(H\ch_2)^2+\frac{1}5(H^2\ch_1)^2\label{eq:337}\\
=& \frac{6}5H^2\ch_1(\frac{3}2H^2\ch_1-H\ch_2-H^3\rk)+\frac{1}5(4H^2\ch_1-10H\ch_2)(H^2\ch_1-2H\ch_2)\label{eq:338}\\
=& \frac{4}5(\frac{3}2(H^2\ch_1)^2-H\ch_2H^3\rk-H^2\ch_1H^3\rk) \nonumber\\
&  +\frac{1}5(7H^2\ch_1-10H\ch_2-2H^3\rk)(H^2\ch_1-2H\ch_2).
\label{eq:339}
\end{align}
By Theorem \ref{thm:BGforslopestableX5} and the assumption that $\nu_{BN}(E)\in (0,\frac{1}2]$, we have $\frac{H^2\ch_1(E)}{H^3\rk(E)}\notin[0,\frac{1}2]$.
\begin{itemize}
\item When $\frac{H^2\ch_1(E)}{H^3\rk(E)} \notin[\frac{1}2,1]$, we have $H^2\ch_1(E)>0$, $H^2\ch_1(E)>H^3\rk(E)$ and $H^2\ch_1(E)= \frac{1}{\nu_{BN}(E)} H\ch_2(E)\geq 2H\ch_2(E)$, hence the equation (\ref{eq:336}) is non-negative.
\item When $\frac{H^2\ch_1(E)}{H^3\rk(E)}\in[\frac{1}2,\frac{3}4]$,  by Theorem \ref{thm:BGforslopestableX5}, the equation (\ref{eq:337}) is non-negative.
\item When $\frac{H^2\ch_1(E)}{H^3\rk(E)}\in[\frac{3}4,\frac{10}{11}]$,  by Theorem \ref{thm:BGforslopestableX5}, $\frac{3}2H^2\ch_1(E)-H\ch_2(E)-H^3\rk(E) \geq 0$ and $\nu_{BN}(E)< \frac{2}5$. Therefore $4H^2\ch_1(E)-10H\ch_2(E)\geq 0$, the equation (\ref{eq:338}) is non-negative.
\item When $\frac{H^2\ch_1(E)}{H^3\rk(E)}\in[\frac{10}{11},1]$,  by Theorem \ref{thm:BGforslopestableX5}, the first term in  equation (\ref{eq:339}) is non-negative. The term $7H^2\ch_1(E)-10H\ch_2(E)-2H^3\rk(E)$ is also non-negative since $2H\ch_2(E)+2H^3\rk(E)\leq 3H^2\ch_1(E)$ by Theorem \ref{thm:BGforslopestableX5}. Therefore, the equation (\ref{eq:339}) is non-negative.
\end{itemize}

As a summary, when $\nu_{BN}(E)\in(0,\frac{1}2]$, we always have $Q_{0,0}(E)\geq 0$.\\

The same argument applies for the case when $\nu_{BN}(E)\in[-\frac{1}2,0)$. In that case, we will have $\Hom(\mathcal O_X,E)=0$ and we can bound the dimension $\hom(\mathcal O_X,E[2])=\hom(E,\mathcal O[1])$ by Lemma \ref{lem:extension}. As pointed out by the referee, we may also consider the derived dual $\mathbb D(E):=E^\vee[1]$. By \cite[Proposition 5.1.3 (b)]{BMT:3folds-BG}, it fits into an exact triangle $\bar E\rightarrow \mathbb D(E) \rightarrow T_0[-1]$ for a Brill-Noether semistable object $\bar E\in \Coh^{0,H}(X)$ and a zero-dimensional torsion sheaf $T_0$. Note that $\ch_1(\bar{E})=\ch_1(\mathbb D(E))=\ch_1(E)$ and $\ch_2(\bar{E})=-\ch_2(E)$ we have $\nu_{BN}(E)\in(0,\frac{1}2]$.  Therefore,
\begin{align*}
Q_{0,0}(E) = Q_{0,0}(\mathbb D(E))=Q_{0,0}(\bar{E})+6\ch_1(E)\ch_3(T_0)\geq 0.
\end{align*}

As for the remaining case that $\nu_{BN}(E)=0=H\ch_2(E)$, we consider the object $\tilde E\coloneqq Cone(\mathcal O_X\otimes \Hom(\mathcal O_X,E)\xrightarrow{can} E)$. If $\tilde {E}$ is $\nu_{\alpha,0,H}$-tilt semistable for some $\alpha>0$, then by Theorem \ref{thm:BGforslopestableX5},  we know that $\frac{H^2\ch_1(\tilde{E})}{H^3\rk(\tilde{E})}\notin(-\frac{1}2,0]$. As $H^2\ch_1(\tilde{E})=H^2\ch_1(E)>0$, we have
\begin{align}
H^3\rk(\tilde E)\geq -2H^2\ch_1(\tilde E).
\label{eq:341}
\end{align}
Otherwise, for each $\delta>0$, $\tilde E$ is destabilized by some $\nu_{\delta^2,\delta,H}$-tilt stable object $F_\delta \hookrightarrow\tilde E$ in $\Coh^{\delta,H}(X)$. We may assume $0<\delta<\frac{1}{2}$ sufficiently small so that $E$ is $\nu_{\delta^2,\delta,H}$-tilt stable. Note that either $\Hom(F_{\delta},E)\neq 0$ or $\Hom(F_{\delta},\mathcal O_X[1])\neq 0$. We have either $\nu_{\delta^2,\delta,H}(F_\delta)\leq \nu_{\delta^2,\delta,H}(\mathcal O[1])$ or 
$\nu_{\delta^2,\delta,H}(F_\delta)\leq \nu_{\delta^2,\delta,H}(E)$. By Theorem \ref{thm:BGforslopestableX5}, when $\delta<\frac{1}2$, we always have $\nu_{\delta^2,\delta,H}(E)<\nu_{\delta^2,\delta,H}(\mathcal O[1])=\delta$. Therefore, $\nu_{\delta^2,\delta,H}(F_\delta)\leq \nu_{\delta^2,\delta,H}(\mathcal O[1])$. Note that the `$=$' can only hold when $F_\delta\simeq \mathcal O[1]$, but then $\Hom(F_\delta,\tilde E)=0$. Therefore,  $\nu_{\delta^2,\delta,H}(F_\delta)<\delta$.

We may assume that $F_\delta$ has the greatest $\nu_{\delta^2,\delta,H}$ slope among all destabilizing subobject of $\tilde E$ in $\Coh^{\delta,H}(X)$. Then for each Harder-Narasimhan factor $E_i$ of $\tilde E$ with respect to $\nu_{\delta^2,\delta,H}$, we have $\nu_{\delta^2,\delta,H}(E_i)<\delta$.  By Lemma \ref{lem:nestedandopen}, each $E_i$ is also $\nu_{\alpha_i,0,H}$-tilt stable for some $\alpha_i>0$ and in addition $\nu_{BN}(E_i)<\delta$.  By Theorem \ref{thm:BGforslopestableX5}, $$\frac{H^2\ch_1(E_i)}{H^3\rk(E_i)}\notin [\frac{-1}{4\delta+2},0].$$  
Or equivalently, $\frac{H^3\rk(E_i)}{H^2\ch_1(E_i)}>-4\delta-2$. When  $\delta$ tends to $0$, we have $H^3\rk(\tilde E)\geq -2H^2\ch_1(\tilde E)$. As (\ref{eq:341}) always holds, we have
\begin{align}
\hom(\mathcal O_X,E)\leq \frac{2}5H^2\ch_1(E)+\rk(E).\label{eq:342}
\end{align}
Recall that by \cite[Proposition 5.1.3 (b)]{BMT:3folds-BG}, there is an exact triangle $\bar E\rightarrow \mathbb D(E)\rightarrow T_0[-1]$ for a Brill-Noether semistable object $\bar E\in \Coh^{0,H}(X)$ and a zero-dimensional torsion sheaf $T_0$. We have 
\begin{align}
\hom(\mathcal O_X,E[2]) & =\hom(E,\mathcal O_X[1]) = \hom (\mathbb D(\mathcal O_X[1]),\mathbb D(E)) \\ & =\hom(\mathcal O_X,\mathbb D(E))=\hom(\mathcal O_X,\bar E) \leq \frac{2}5H^2\ch_1(E)-\rk(E).\label{eq:343}
\end{align}
By Hirzebruch-Riemann-Roch (\ref{eq:342}) and (\ref{eq:343}), we have 
$$\ch_3(E)+\frac{5}6H^2\ch_1(E)\leq \chi(E)\leq \hom(\mathcal O_X,E)+\hom(\mathcal O_X,E[2])=\frac{4}5H^2\ch_1(E).$$
Therefore, $\ch_3(E)<0$ and $Q_{0,0}(E)>0$. 

In any case of $\nu_{BN}(E)$, we always have $Q_{0,0}(E)\geq 0$.
\end{proof}

\section{Clifford type inequality for curves {$C_{2,2,5}$}}
\label{sec:clf}

The generalized Clifford index theorem  for curves, \cite[Theorem 2.1]{Brambila-Paz1995Geography} states that for any semistable vector bundle $F$ over a smooth curve $C$ with rank $r$ and slope $\mu\in[0,g]$, where $g$ is the genus of the curve, the following bound holds:
$$h^0(F)/r\leq 1+\frac{\mu}2.$$

The main purpose of this section is to set up the following stronger Clifford type inequality for the curve $C_{2,2,5}$, which is the complete intersection of two quadratic hypersurfaces and a quintic hypersurface in $\mathbf {P}_\C^4$.
\begin{Prop}
Let $F$ be a semistable vector bundle on a smooth curve $\cer$ with rank $r$ and slope $\mu\in(0,10]\cup [30,40]$, then we have the following bounds for the dimension of global sections of $h^0(F)$:
\begin{align}\label{eq:clf}
h^0(F)/r\leq \begin{cases}
\frac{40}{41}+\frac{\mu}{41}, & \text{when } \mu\in(0,2);\\
\max\{\frac{24}{25}+\frac{4}{125}\mu,\frac{241}{248}+\frac{33}{1240}\mu\}, & \text{when } \mu\in[2,\frac{5}2);\\
\max\{\frac{12}{13}+\frac{3}{65}\mu,\frac{295}{306}+ \frac{19}{612}\mu\},& \text{when } \mu\in[\frac{5}2,\frac{10}3);\\
\max\{\frac{4}{5}+\frac{2}{25}\mu,\frac{193}{204}+ \frac{7}{170}\mu\},  & \text{when } \mu\in[\frac{10}3,5);\\
\max\{\frac{1}{5}\mu,\frac{55}{62}+ \frac{9}{124}\mu\}, & \text{when } \mu\in[5,10];\\
\frac{2}{5}\mu-4, & \text{when } \mu\in[30,37];\\
\frac{11}{15}\mu-\frac{49}{3}, & \text{when } \mu\in[37,40].\\
\end{cases}
\end{align}
\label{prop:clfboundsforC225}
\end{Prop}

The bound listed above is the best result we can prove so far. As for the purpose to prove Proposition 5.2, when $\mu\in[2,10]$, we only need the following weaker but neat bound.

When $\mu\in[2,5)$,  the right hand side is always less than or equal to
$\frac{7\mu}{120}+\frac{109}{120}$.

When $\mu\in[5,10]$,  the right hand side is always less than or equal to $\frac{3\mu}{20}+\frac{1}2$. 

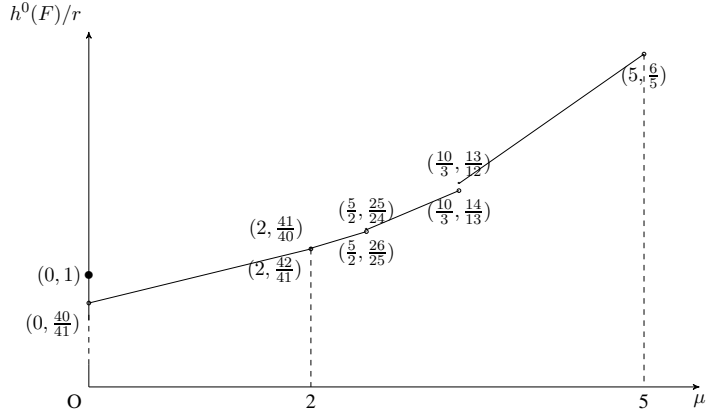
\begin{figure}[hbt!]
\begin{center}
\scalebox{0.7}{
\begin{tikzpicture}[y=600pt,x=60pt,
circ/.style={shape=circle, inner sep=0.7pt, draw, node contents=}]

\newcommand\UP{0.9};

\draw[->] (0,0) -- (5.5,0) node[below] {$\mu$};

\draw (0,0) node [below left] {O}  -- (0,0.02);
\draw[dashed] (0,0.02)--(0,0.07);
\draw[->] (0,0.06) --  (0,0.32) node[above left] {$h^0(F)/r$};

\node[left] at (0,1-\UP) {$(0,1)$};
\node at (0,1-\UP) {$\bullet$};

 \draw node at (0,40/41-\UP) [circ, label=below left:{$(0,\frac{40}{41})$}];
 \draw (0,40/41-\UP)--(2,42/41-\UP);
 \draw node at (2,42/41-\UP) [circ, label=below left:{$(2,\frac{42}{41})$}];
 
 \draw[dashed] (2,42/41-\UP)-- (2,0) node [below]{2};
 
 \draw node at (2,41/40-\UP) {$\cdot$} node[above left] at (2,41/40-\UP)  {$(2,\frac{41}{40})$};
 \draw (2,41/40-\UP)--(5/2,26/25-\UP);
 \draw node at (5/2,26/25-\UP) [circ, label=below:{$(\frac{5}2,\frac{26}{25})$}];
 
 \draw node at (5/2,25/24-\UP){$\cdot$} node[above] at (5/2,25/24-\UP) {$(\frac{5}2,\frac{25}{24})$};
 \draw (10/3,14/13-\UP)--(5/2,25/24-\UP);
 \draw node at (10/3,14/13-\UP) [circ,label=below:{$(\frac{10}3,\frac{14}{13})$}];
 
 \draw node at (10/3,13/12-\UP){$\cdot$} node[above] at (10/3,13/12-\UP) {$(\frac{10}3,\frac{13}{12})$};
 \draw (10/3,13/12-\UP)--(5,6/5-\UP);
 \draw node at (5,6/5-\UP) [circ,label=below:{$(5,\frac{6}{5})$}];
 \draw[dashed] (5,6/5-\UP)-- (5,0) node [below]{5};
\end{tikzpicture}
}

\end{center}
\caption{Bounds for $h^0(F)/r$ when $\mu\in(0,5)$.} \label{fig:clfbound}
\end{figure}

\begin{Rem}
The inequality is sharp for some values of the slope  $\mu$. When $\mu\in (0,2)$, this result is the sharp bound as shown in \cite[Proposition 2.1]{Lange2015g5}. When $\mu=2,\frac{5}2,\frac{10}3,5,10,30,40$, this bound is also for sure to be sharp. 

 To express those vector bundles with sharp bound more precisely, we let $S_{2,2}$ be  a smooth complete intersection of two quadratic hyper-surfaces such that $S_{2,2}$ contains $C_{2,2,5}$ in $\mathbf{P}^4_{\C}$. Note that $S_{2,2}$ is a del Pezzo surface with degree $4$, it can be viewed as a projective plane blown-up at $5$ points. Denote $\ell_0$ as the pull-back of a line in $\mathbf{P}^2$ and $\ell_1$ as one of the exceptional lines.
  
When $\mu=\frac{10}{n+1}$ for $n=1,2,3,4$, we may achieve the maximum $h^0(F)/r$ by letting $F=E_n|_{C_{2,2,5}}$ where $E_n$ is a vector bundle on $S_{2,2}$ defined as the cokernel of the map
$$\mathcal O_{S_{2,2}}(-nH)\xrightarrow{can}\mathcal O_{S_{2,2}}\otimes \Hom(\mathcal O_{S_{2,2}}(-nH),\mathcal O_{S_{2,2}})^*. $$

When $\mu=10$, one may let $F=\mathcal O_{S_{2,2}}(\ell_0-\ell_1)|_{C_{2,2,5}}$. When $\mu=40$, one may let $F=\cO(2H)|_{C_{2,2,5}}$.
\label{rem:boundforclf}
\end{Rem}

Back to the proof for the Proposition \ref{prop:clfboundsforC225}, it is enough to prove the statement for stable vector bundles. We denote the inclusion map by $\iota:C_{2,2,5}\hookrightarrow S_{2,2}$. In this section, we write $H$ for $[\mathcal O_{S_{2,2}}(1)]$ and only use stability conditions on $S_{2,2}$ with polarization $H$. As $h^0(F)=h^0(\iota_*F)$, we will always consider the dimension of global sections on $\iota_*F$ in $D^b(S_{2,2})$ instead of $F$. The following statement is standard:

\begin{Lem}
Let $F$ be a stable vector bundle on $C_{2,2,5}$, then $\iota_*F$ is $\nu_{\alpha,0,H}$-tilt stable for $\alpha\gg 0$.
\label{lem:iotaFstable}
\end{Lem}

Following the strategy in \cite{Fe17} and \cite{FL:clford}, we will compute $h^0(\iota_*F)$ by considering the Harder-Narasimhan factors of $\iota_*F$ with respect to $\nu_{BN}$.

\begin{Lem}[{\cite[Proposition 3.4 (a)]{Fe17}}]
For each object $E\in \Coh^{0,H}(S_{2,2})$ that is $\nu_{\alpha,0,H}$-tilt stable for some $\alpha>0$, there exists $\delta>0$ such that  there is a Harder-Narasimhan filtration for $E$ with respect to $\nu_{\alpha,0,H}$ for any $0<\alpha<\delta$:
$$0=E_0\subset E_1\subset \dots E_m=E.$$
In particular, each factor $F_i=E_i/E_{i-1}$ is Brill-Noether semistable. The slope is decreasing $\nu_{BN}^+(E)\coloneqq \nu_{BN}(F_1)\geq\dots\geq \nu_{BN}(F_n)=:\nu_{BN}^-(E)$.
\label{lem:HNforbnstab}
\end{Lem}

The original statement in \cite[Proposition 3.4 (a)]{Fe17} only states for K3 surfaces. But the argument only needs that there are finitely many possible classes for semistable factors, which is due to Bogomolov inequality. So it holds for every surface.

The geometric stability conditions on $S_{2,2}$ with polarization $H$ is slightly larger than that ensured by the Bogomolov inequality. In particular, we may choose the parameter $\alpha\leq \frac{\beta^2}2$.

\begin{Def} \label{def:gcurvep2}
Let $\gamma:\R\rightarrow \R$ be a $1$-periodic function. When $x\in [0,1]$, 
\[ \gamma(x) \coloneqq \begin{cases}
\frac{1}{2}x^2 - \frac{1}2 x +\frac{1}4, & \text{if $x\in (0,1)$;} \\
0, & \text{if $x = 0$ or $1$}.
\end{cases}\]
Let $\Gamma(x)\coloneqq \frac{1}{2}x^2-\gamma(x)$. 
\end{Def}

\begin{Obs}
For any torsion-free $\mu_H$-slope stable object $E$, we have
$${\Gamma}\left(\frac{H\ch_1(E)}{H^2\rk(E)}\right)\geq \frac{\ch_2(E)}{H^2\rk(E)}.$$
The data $(\Coh^{0,H}(S_{2,2}),Z_{\alpha,\beta,H})$ parameterized by $\{(\alpha,\beta)\in\R^2|\alpha>\Gamma(\beta)\}$ form a continuous family of stability conditions on $D^b(S_{2,2})$.
\label{obs:geostabonS22}
\end{Obs}

\begin{proof}
The inequality for Chern characters of $\mu_H$-slope stable objects with $\frac{H\ch_1(E)}{H^2\rk(E)}\notin \Z$ is by computing 
$$0\geq -\hom(\mathcal O_{S_{2,2}}(\floor{\frac{H\ch_1(E)}{H^2\rk(E)}+1}H),E[1])=\chi(E(-\floor{\frac{H\ch_1(E)}{H^2\rk(E)}+1}H)).$$
The stability condition is then a standard construction as that in \cite{Bridgeland:K3} or the framework in \cite[Section 2]{Dulip-Toda:bgtostab}.
\end{proof}

\begin{Rem}
It is worth to mention that Bertram's nested wall theorem (Lemma \ref{lem:nestedandopen} (b)) still holds for $(\alpha,\beta)$  on the wall such that $\alpha>\Gamma(\beta)$, but one needs to be careful that in this case every point $(\alpha',\beta')$  on the line segment between $(\alpha_0,\beta_0)$ and $(\alpha,\beta)$ should also satisfy $\alpha'>\Gamma(\beta')$.
\label{rem:nestedwallforS22}
\end{Rem}

The following lemma explains that we can estimate the dimension of global sections for each Brill-Noether semistable factor.

\begin{Lem}
\label{lem:p2length}
Let $F\in \Coh^{0,H}(S_{2,2})$ be a Brill-Noether semistable object.
Then 
\begin{align*}
\hom(\mathcal O_{S_{2,2}}, F) \begin{cases}
= \rk(F) + \frac{1}2H\ch_1(F)+\ch_2(F), & \text{ when }-\frac{1}2<\nu_{BN}(F) <+\infty; \\
\leq \rk(F) +\frac{1}{4n} H\ch_1(F), & \text{ when }\nu_{BN}(F)=-\frac{n}2, n\in \Z_{>0};\\
\leq \rk(F) +\frac{2n+1}{4n^2+4n+2} H\ch_1(F)\\ \;\;\;\;+\frac{1}{2n^2+2n+1} \ch_2(F), & \text{ when }-\frac{n+1}2<\nu_{BN}(F)<-\frac{n}2, n\in \Z_{>0}.
\end{cases} 
\end{align*} 
\end{Lem}
\begin{proof}
When $-\frac{1}2<\nu_{BN}(F) <+\infty$, since $\mathcal O_{S_{2,2}}(-H)[1]$ is Brill-Noether stable with slope $$\nu_{BN}(\mathcal O_{S_{2,2}}(-H)[1])=-\frac{1}2,$$
we have $$\hom(\mathcal O_{S_{2,2}}, F[1+i])=\hom(F,\mathcal O_{S_{2,2}}(-H)[1-i]) =0$$ for $i\geq 0$. Since $O_{S_{2,2}}[1]$ is $\nu_{\alpha,0,H}$-tilt stable for $\alpha>0$ and has slope 
$$\nu_{BN}(\mathcal O_{S_{2,2}}[1])=+\infty,$$
we have $$\hom(\mathcal O_{S_{2,2}},F[-1-i])=\hom(\mathcal O_{S_{2,2}}[1+i],F)=0$$
for $i\geq 0$. Therefore,
$$\hom(\mathcal O_{S_{2,2}}, F) =\chi(F)
= \rk(F) + \frac{1}2H\ch_1(F)+\ch_2(F).$$

When $\nu_{BN}(F)\leq -\frac{1}2$, there exists $(\alpha,\beta)$ on the line through $p_H(F)$ and $(0,0)$ such that $\alpha>\frac{1}2\beta^2$. By Lemma \ref{lem:nestedandopen}, the object $F$ is in $\Coh^{\beta,H}(S_{2,2})$ and $\nabH$-tilt semistable with $\nu_{BN}(F)=\nabH(F)=\frac{\alpha}{\beta}=\nabH(\cO_{S_{2,2}})$. The object  $$\tilde{F}\coloneqq Cone(\mathcal O_{S_{2,2}}\otimes \Hom(\mathcal O_{S_{2,2}},F) \xrightarrow{can} F)$$
is in $\Coh^{0,H}(S_{2,2})$ and therefore also in $\Coh^{\beta,H}(S_{2,2})$ by Lemma \ref{lem:nestedandopen}. In particular, the object $\tilde F$ 
is $\nabH$-tilt semistable with slope $\nabH(\tilde F)=\nu_{BN}(F)$. Since $\db_H(\tilde F)\geq 0$, we have
\begin{align*}
0\leq & (H\ch_1(F))^2-2H^2\ch_2(F)(\rk(F)-\hom(\mathcal O_{S_{2,2}},F))\\
\implies \hom(\mathcal O_{S_{2,2}},F) \leq & \rk(F)-\frac{(H\ch_1(F))^2}{2H^2\ch_2(F)}=\rk(F) -\frac{1}{8\nu_{BN}(F)} H\ch_1(F).
\end{align*}
This verifies the case when  $\nu_{BN}(F)=-\frac{n}2$ for $n\in \Z_{>0}$.

When $-\frac{n+1}2<\nu_{BN}(F)<-\frac{n}2, n\in \Z_{>0}$, note that $\mathcal O_{S_{2,2}}(-(n+1)H)[1]$ is in $\Coh^{\beta,H}(S_{2,2})$ and $\nabH$-tilt stable with slope 
$$\nabH(\mathcal O_{S_{2,2}}(-(n+1)H)[1])\rightarrow \nu_{BN}(\mathcal O_{S_{2,2}}(-(n+1)H)[1])=-\frac{n+1}2,$$
when $\beta\rightarrow 0$. We have  $$\hom(\mathcal O_{S_{2,2}}(-nH), \tilde F[1+i])=\hom(\tilde F,\mathcal O_{S_{2,2}}(-(n+1)H)[1-i]) =0$$ for $i\geq 0$. Since $O_{S_{2,2}}(-nH)[1]$ is in $\Coh^{\beta,H}(S_{2,2})$ and $\nabH$-tilt stable for with  slope 
$$\nabH(\mathcal O_{S_{2,2}}(-nH)[1])\rightarrow\nu_{BN}(\mathcal O_{S_{2,2}}(-nH)[1])=-\frac{n}2,$$
when $\beta\rightarrow 0$. We have $$\hom(\mathcal O_{S_{2,2}}(-nH),\tilde F[-1-i])=\hom(\mathcal O_{S_{2,2}}(-nH)[1],\tilde F[-i])=0$$for $i\geq 0$. Therefore,
\begin{align*}
0\leq & \hom(\mathcal O_{S_{2,2}}(-nH),\tilde F)=\chi(\tilde F(nH))\\ =& \ch_2(F) +(n+\frac{1}2)H\ch_1(F)+\frac{(n^2+n)H^2+2}2(\rk(F)-\hom(\mathcal O_{S_{2,2}},F))\\
\implies & \hom(\mathcal O_{S_{2,2}},F))  \leq  \rk(F) +\frac{2n+1}{4n^2+4n+2} H\ch_1(F)+\frac{1}{2n^2+2n+1} \ch_2(F).
\end{align*}
We finish the claim for all cases.
\end{proof}

The following property decides the bounds for the Brill-Noether slopes of each Harder-Narasimhan factors of $\iota_*F$.

\begin{Prop}
Let $F$ be a slope stable vector bundle on $C_{2,2,5}$ with slope $\mu(F)\in (0,10]\cup [30,40]$, then 
\begin{itemize}
\item $\iota_*F$ is Brill-Noether semistable when $\mu\leq 30-20\sqrt 2;$ 
\item $\nu_{BN}^+(\iota_*F)\leq \begin{cases}
\frac{1}2-\frac{5}\mu, & \text{ when }\mu\in (30-20\sqrt 2,10];\\
\frac{3}2-\frac{25}\mu, & \text{ when }\mu\in [30,10+20\sqrt 2];\\
\frac{23\mu-610}{12\mu-140}, & \text{ when }\mu\in [10+20\sqrt 2,39];\\
\frac{\mu}8-4, & \text{ when }\mu\in [39,40].
\end{cases} $
\item $\nu_{BN}^-(\iota_*F)\geq \begin{cases}
\frac{490-9\mu}{2\mu-200}, & \text{ when }\mu\in (30-20\sqrt 2,10];\\
\frac{450-7\mu}{2\mu-200}, & \text{ when }\mu\in [30,10+20\sqrt 2];\\
-\frac{3}2, & \text{ when }\mu\in [10+20\sqrt 2,40].
\end{cases} $
\end{itemize}
\label{prop:boundsforBNofiotaF}
\end{Prop}

To prove the proposition, we need estimate the first wall of $\iota_*F$.
\begin{Lem}\label{lem:fwforiotaF}
Adopt the notations in Proposition \ref{prop:boundsforBNofiotaF}, if $\mu\in(0,30-20\sqrt 2]$, then $\iota_*F$ is Brill-Noether semistable.
Otherwise, suppose $\iota_*F$ becomes strictly $\nu_{\alpha,0,H}$-tilt semistable for some $\alpha>0$. Then,
\begin{equation*}
\alpha\leq \begin{cases}
\frac{3\mu}{20}-\frac{\mu^2}{400}-\frac{1}4, & \text{ when } \mu\in(30-20\sqrt 2,10]; \\
\frac{\mu}5-\frac{\mu^2}{400}-\frac{5}4, & \text{ when } \mu\in[30,10+20\sqrt 2];\\
\frac{3\mu}{20}-3, & \text{ when } \mu\in(10+20\sqrt 2,40]. 
\end{cases}
\end{equation*}
\end{Lem}

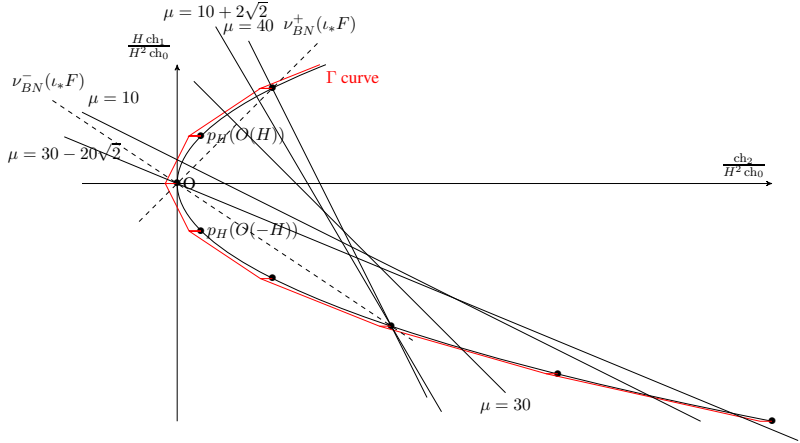
\begin{figure}[hbt!]
\begin{center}
\scalebox{0.6}{

\begin{tikzpicture}[x=30pt,y=30pt]
\tikzset{%
    add/.style args={#1 and #2}{
        to path={%
 ($(\tikztostart)!-#1!(\tikztotarget)$)--($(\tikztotarget)!-#2!(\tikztostart)$)%
  \tikztonodes},add/.default={.2 and .2}}
}

\newcommand\XA{1}

\draw[->] (-2,0) -- (12.5,0)node[above left] {$\frac{\ch_2}{H^2\ch_0}$};

\draw[->] (0,-5)-- (0,0) node [right] {O} --  (0,2.5) node[above left] {$\frac{H\ch_1}{H^2\ch_0}$};

\draw[rotate=-90] (0,0) parabola (5,12.5);
\draw[rotate=-90] (0,0) parabola (-2.5,3.125);

\coordinate (O) at (0,0);
\coordinate(O1) at (0.5*\XA,1);
\coordinate(O2) at (2,2);
\coordinate(O-1) at (0.5*\XA,-1);
\coordinate(O-2) at (2,-2);
\coordinate(O-3) at (4.5,-3);
\coordinate(O-4) at (8,-4);
\coordinate(O-5) at (12.5,-5);
\node at (O1) {$\bullet$};
\node at (O) {$\bullet$};
\node at (O2) {$\bullet$};
\node at (O-1) {$\bullet$};
\node at (O-2) {$\bullet$};
\node at (O-3) {$\bullet$};
\node at (O-4) {$\bullet$};
\node at (O-5) {$\bullet$};
\node[right] at (O1) {$p_H(O(H))$};
\node[right] at (O-1) {$p_H(O(-H))$};

\draw[red] (3,2.5) node[below right] {$\Gamma$ curve}--(7/4,2)--(2,2)--(7/4,2)--(1/4,1)--(1/2,1)--(1/4,1)--(-1/4,0)--(0,0)--(-1/4,0)--(1/4,-1)--(1/2,-1)--(1/4,-1);
\draw[red](1/4,-1)--(7/4,-2)--(O-2);
\draw[red](7/4,-2)--(17/4,-3)--(O-3);
\draw[red](17/4,-3)--(31/4,-4)--(O-4);
\draw[red](31/4,-4)--(49/4,-5)--(O-5);

\draw [add= 0.3 and 0.2] (O-3) to (O2) node[above]{$\mu=40$};
\draw [add= 0.13 and 0.18] (1,1.5) to (6,-3.5) node[below]{$\mu=30$};
\draw [add= 0.4 and 0.5,dashed] (O) to (O2) node[above]{$\nu^+_{BN}(\iota_*F)$};
\draw [add= 0.1 and 0.6,dashed] (O-3) to (O) node[above]{$\nu^-_{BN}(\iota_*F)$};

\draw [add= 0.6 and 1.1] (O-3) to (2.742,0) node[above]{$\mu=10+2\sqrt{2}$};

\draw [add= 0.1 and .2] (10,-4.5) to (0,0.5) node[above right]{$\mu=10$};
\draw [add= 0.1 and .2] (11.863,-4.914) to (0,0) node[below]{$\mu=30-20\sqrt 2$};
\end{tikzpicture}
}
\end{center}
\caption{The potential first wall of $\iota_*F$.} \label{fig:firstwallforiotaF}
\end{figure}
\begin{proof}
We write $r$ for the rank of $F$, the Chern characters of $\iota_*F$ are as follows:
$$(\ch_0(\iota_*F),\ch_1(\iota_*F),\ch_2(\iota_*F))=(0,5rH,-r\frac{(5H)^2}{2}+r\mu=r(\mu-50)).$$
Let $0\rightarrow F_2\rightarrow \iota_*F\rightarrow F_1\rightarrow 0$ in $\Coh^{0,H}(S_{2,2})$ be the destabilizing sequence with respect to $\nu_{\alpha,0,H}$, then there is an exact sequence in $\Coh (S_{2,2})$.
\begin{center}
\begin{tikzcd}[row sep=tiny]
0\arrow[r] & \HH^{-1}(F_1) \arrow[r] 
& F_2 \arrow[r] & \iota_*F \arrow[r] & \HH^{0}(F_1) \arrow[r] & 0. \\
\text{rank} & s & s & 0 & 0 \\
\ch_1 & D_1 & D_2 & 5rH & 5aH
\end{tikzcd}
\end{center}

If $s=0$, then $\HH^{-1}(F_1)=0$ as it is torsion free. Since $F_2$ and $\iota_*F$ have the same $\nu_{\alpha,0,H}$ slope, we must have $\ch(\iota_*F) = k\ch(F_2)$ for some real number $k > 0$, this will violate the stability assumption on $F$. Thus, we may assume $s \neq 0$. 

Let $T(F_2)$ be the maximal torsion subsheaf of $F_2$. Without loss of generality  we may assume that it is supported on $C_{2,2,5}$ with $\ch_1(T(F_2)) = 5tH$. Since $F$ is of rank $r$, to make the sequence exact at the term $\iota_*F$, we must have 
\begin{equation*}
r-a \leq \rank \big(\iota_*T(F_2)\big) + \rank\big( F_2/T(F_2)\big) = s+t. 
\end{equation*}
Therefore, 
\begin{equation}\label{eq:diffofF}
\frac{H\ch_1\big(F_2/T(F_2)\big)}{sH^2} - \frac{H\ch_1\big(H^{-1}(F_1)\big)}{sH^2} = \frac{D_2H-5tH^2 -D_1H}{sH^2} = \frac{5r-5a-5t}{s} \leq 5.
\end{equation}
By Lemma \ref{lem:nestedandopen}, the objects $F_1$ and $F_2$ are $\nu_{\alpha',\beta',H}$-tilt semistable of the same phase as $\iota_*E$ for any $(\alpha',\beta')$ along the wall $W$ through $(\alpha,0)$ with slope $1/\nu_{BN}(\iota_*F)=1/(\frac{\mu}{20}-\frac{5}2)$. Let $(\alpha_1,\beta_1)$ and $(\alpha_2,\beta_2)$ be the intersection points of $W$ and the curve $\Gamma$, more precisely, 
\begin{align*}
\beta_1\coloneqq \max\{y<0|\Gamma(y)\geq (\frac{\mu}{20}-\frac{5}2)y+\alpha\};\;\;\;
\beta_2\coloneqq \min\{y>0|\Gamma(y)\geq (\frac{\mu}{20}-\frac{5}2)y+\alpha\}.
\end{align*}
By Lemma \ref{lem:nestedandopen} and Remark \ref{rem:nestedwallforS22}, the object $F_2/T(F_2)$ is in the heart $\Coh^{\beta_2-\epsilon,H}X$  and $\HH^{-1}(F_1)$ is in the heart $\Coh^{\beta_1+\epsilon,H}X$  for sufficiently small $\epsilon > 0$. Thus by definition of the tilting heart and (\ref{eq:diffofF}),
we have 
\begin{align}
\beta_2-\beta_1\leq 5.
\end{align}
Now we have reduced the first wall through $(\alpha,0)$ to an elementary computation. Note that the line through $(\Gamma(\beta_2),\beta_2)$ and $(\Gamma(\beta_2-5),\beta_2-5)$ always has slope $1/(\beta_2-\frac{5}2)$. 
\begin{itemize}
\item When $0<\mu\leq 10$, note that the line through $(\Gamma(\frac{\mu}{20}),\frac{\mu}{20})$ and $(\Gamma(\frac{\mu}{20}-5),\frac{\mu}{20}-5)$ has equation 
$$X-\frac{\mu}{40}+\frac{1}4=(\frac{\mu}{20}-\frac{5}2)(Y-\frac{\mu}{20}).$$
It passes through the point $(\frac{3\mu}{20}-\frac{\mu^2}{400}-\frac{1}4,0).$ The object $\iota_*F$ cannot become strictly $\nu_{\alpha,0,H}$-tilt semistable for any $\alpha>\max\{0,\frac{3\mu}{20}-\frac{\mu^2}{400}-\frac{1}4\}$.
\item When $30\leq \mu \leq 40$, two different types of lines are possible to be the first wall as shown in Figure \ref{fig:firstwallforiotaF}. We list their equations as follows:
\begin{enumerate}
\item The line through $(\Gamma(-3),-3)$ with slope $1/(\frac{\mu}{20}-\frac{5}2)$ has equation 
$(\frac{\mu}{20}-\frac{5}2)(Y+3)=X-\frac{9}2.$
It passes through $(\frac{3\mu}{20}-3,0)$.
\item 
The line through $(\Gamma(\frac{\mu}{20}),\frac{\mu}{20})$ and $(\Gamma(\frac{\mu}{20}-5),\frac{\mu}{20}-5)$  has equation 
$X-\frac{3\mu}{40}+\frac{5}4=(\frac{\mu}{20}-\frac{5}2)(Y-\frac{\mu}{20}).$
It passes through $(\frac{\mu}{5}-\frac{\mu^2}{400}-\frac{5}4,0)$.
\end{enumerate}
The object $\iota_*F$ cannot become strictly $\nu_{\alpha,0,H}$-tilt semistable for any $\alpha>\max\{\frac{3\mu}{20}-3,\frac{\mu}{5}-\frac{\mu^2}{400}-\frac{5}4\}$.
\end{itemize}
\end{proof}

\begin{proof}[Proof for Proposition \ref{prop:boundsforBNofiotaF}]
Suppose $\iota_*F$ is not Brill-Noether semistable. Let the Harder-Narasimhan filtration of $\iota_*F$ with respect to $\nu_{BN}$ be 
$$0=E_0\subset E_1\subset \dots \subset E_{m-1}\subset E_m=\iota_*F.$$

Note that $\rk(E_1) > 0$, since otherwise for any $\alpha>0$,
$$\nu_{\alpha,0}(E_1)\geq \nu_{BN}(E_1)>\nu_{BN}(\iota_*F)=\nu_{\alpha,0}(\iota_*F),$$
this is not possible as $E_1$ is a subobject of $\iota_*F$ in $\Coh^{0,H}(S_{2,2})$. 

The line through $p_H(E_1)$ with slope $1/\nu_{BN}(\iota_*F)$ passes through $(\alpha,0)$ for some $\alpha\leq 0$ or satisfying the inequality in Lemma \ref{lem:fwforiotaF}. Together with the constrain that $\rk(E_1)>0$, the slope $\nu_{BN}(E_1)$ can only achieve maximum when both $\alpha$ and $\beta_2$ reach their maximums and $\frac{H\ch_1(E_1)}{H^2\rk(E_1)}=\beta_2$. 
It is a direct computation that 
\begin{align*}
\beta_2=\begin{cases}
\frac{\mu}{20}, & \text{ when }\mu\in (30-20\sqrt 2,10] \cup [30,10+20\sqrt 2];\\
\frac{3\mu-35}{80-\mu}, & \text{ when }\mu\in [10+20\sqrt 2,39];\\
2, & \text{ when }\mu\in [39,40].
\end{cases}
\end{align*}
To compute $\nu_{BN}(E_1)$, the only special case is when $\mu\in[39,40]$, in this case $\beta_1=-3$,  we have $\frac{\ch_2(E_1)}{H^2\rk(E_1)}\leq \frac{\mu}4-8$. The other cases are by computing $\frac{\Gamma(\beta_2)}{\beta_2}$ directly.

As for the $\nu_{BN}(E_m/E_{m-1})$, one may use the same argument and reduce it to the computation of $\frac{\Gamma(\beta_1)}{\beta_1}$.
\end{proof}

Define the function $\clubsuit \colon (x,y) \in \mathbb{H} = \mathbb{R}  \times \mathbb{R}^{>0} \rightarrow \mathbb{R}^{>0}$ as follows \begin{align*}
\clubsuit  (x,y)\coloneqq\begin{cases}
 \frac{y}2+x, & \text{ when } \frac{x}y>-\frac{1}2; \\
\frac{y}{4n} , & \text{ when }\frac{x}y=-\frac{n}2, n\in \Z_{>0};\\
\frac{2n+1}{4n^2+4n+2} y +\frac{1}{2n^2+2n+1}x, & \text{ when }-\frac{n+1}2<\frac{x}y<-\frac{n}2, n\in \Z_{>0}.
\end{cases} 
\end{align*} 
Note that the value of $\clubsuit$ is always positive, therefore well defined.

\begin{Lem} \label{lem:polygonmax}
	Let $O=(0,0)$ be the origin, let  $P=(x_p,y_p)$ and $Q=(x_q,y_q)$ be two points on $\H$ such that $\frac{x_p}{y_p}<\frac{x_q}{y_q}$ and $y_p>y_q$. Consider all collections of points $P_0=O$, $P_1$,  $\dots$, $P_n=P$ in the triangle $OQP$ such that $P_0P_1\dots P_nP_0$ forms a convex polygon, then the sum 
    \begin{equation}
    \label{eq:sum}\sum^{k=n}_{k=1}\clubsuit(\overrightarrow{P_{k-1}P}_k)
    \end{equation}
   can achieve its maximum when either $n=1$ or $2$. In addition, when $n=2$, the point $P_1=(x_1,y_1)$ can be chosen on the line segment $OQ$ ($QP$, respectively) unless $\frac{x_1}{y_1}=-\frac{n}2$ ($\frac{x_p-x_1}{y_p-y_1}=-\frac{n}2$, respectively) for some $n\in\Z_{>0}$.
\end{Lem}

\begin{proof}
Consider the following toy model on the left in Figure \ref{fig:toymodelforclubfunction}: $y_c>y_b>y_a$ and $AC//A'C'$. We allow $A'$ to move alone the line segment $AB$ ($C'$ moves along $BC$ accordingly so that $AC//A'C'$).
Note that the function $\clubsuit(\overrightarrow{AA'})+\clubsuit(\overrightarrow{A'C'})+\clubsuit(\overrightarrow{C'C})$ changes linearly with respect to the length of $AA'$, it can achieve maximum when either $A'=C'=B$ or both $A'=A$ and $C'=C$.

Therefore, to achieve the maximum of (\ref{eq:sum}) we may remove extra $P_i$'s when $n>2$. 
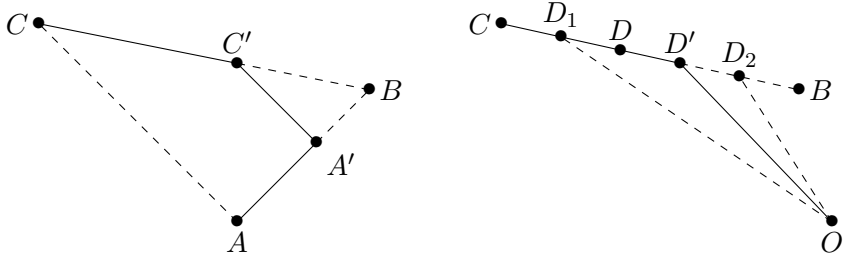
\begin{figure} [hbt!]
	\begin{centering}
		
		\begin{tikzpicture}[x=25pt,y=25pt]
        \tikzset{%
    add/.style args={#1 and #2}{
        to path={%
 ($(\tikztostart)!-#1!(\tikztotarget)$)--($(\tikztotarget)!-#2!(\tikztostart)$)%
  \tikztonodes},add/.default={.2 and .2}}
}
		
		\coordinate (A) at(0,0);
		\coordinate (B) at(2,2);
		\coordinate (C) at(-3,3);

        \draw [add= 0 and -0.4] (A) to (B) coordinate (A1);
        \draw [add= 0 and -0.4] (C) to (B) coordinate (C1);
        \draw (A1) -- (C1);
        
        \node at (A) {$\bullet$};
        \node at (A1) {$\bullet$};
        \node at (C1) {$\bullet$};
        \node at (B) {$\bullet$};
        \node at (C) {$\bullet$};
        
        \node[below] at (A) {$A$};
        \node[right] at (B) {$B$};
        \node[left] at (C) {$C$};
        \node[below right] at (A1) {$A'$};
        \node[above] at (C1) {$C'$};
        \draw[dashed] (A1)-- (B);
        \draw[dashed] (C1)-- (B);
        \draw[dashed] (A)-- (C);
        
		\coordinate (A) at(9,0);
        \coordinate (B) at(8.5,2);
		\coordinate (C) at(4,3);

        \draw [add= -0.2 and -0.8] (C) to (B) coordinate (D1);
        \draw [add= -0.4 and -0.6] (C) to (B) coordinate (D);
        \draw [add= 0 and -0.4] (C) to (B) coordinate (D0);
        \draw [add= -0.8 and -0.2] (C) to (B) coordinate (D2);

        \node at (A) {$\bullet$};
        \node at (B) {$\bullet$};
        \node at (C) {$\bullet$};
        \node at (D1) {$\bullet$};
        \node at (D) {$\bullet$};
        \node at (D0) {$\bullet$};
        \node at (D2) {$\bullet$};
        
        \node[below] at (A) {$O$};
        \node[right] at (B) {$B$};
        \node[left] at (C) {$C$};
        \node[above] at (D1) {$D_1$};
        \node[above] at (D) {$D$};
        \node[above] at (D2) {$D_2$};
        \node[above] at (D0) {$D'$};
        
        \draw[dashed] (A)-- (D1);
        \draw[dashed] (A)-- (D2);
        \draw[dashed] (D0)-- (B);
        \draw (A)--(D0);
        \end{tikzpicture}
		
		\caption{Replacing extra edges and moving the vertex.}
		
		\label{fig:toymodelforclubfunction}
		
	\end{centering}
	
\end{figure}

Consider the toy model on the right in Figure \ref{fig:toymodelforclubfunction}: $y_c>y_b$ and $D$ is on the line segment of $BC$ such that $\frac{x_d}{y_d}\notin \frac{1}2\Z_{<0}$. Then by the definition of $\clubsuit$, there exists $D_1$ and $D_2$ on the line segment of $BD$ and $CD$ respectively such that $\frac{x_{d_i}}{y_{d_i}}\in \frac{1}2\Z_{<0}$ or $D_1=C$ ($D_2=B$), and for any point $D'$ on the line segment $D_1D_2$ ($D'\neq D_1$ or $D_2$), the function $\clubsuit(\overrightarrow{OD'})$ is computed with the same coefficients as that of $\clubsuit(\overrightarrow{OD})$.  Note that 
the function $\clubsuit(\overrightarrow{OD'})+\clubsuit(\overrightarrow{D'C})$ changes linearly with respect to the length of $D_1D'$ when $D\neq D_1$ or $D_2$, and is upper semi-continuous at $D_1$ and $D_2$, it can achieve the maximum when either $D'=D_1$ or $D_2$.

Back to the case of the lemma when $n=2$, we may always adjust the position of $P_1$ so that it satisfies the requirements in the statement.
\end{proof}

\begin{proof}[Proof of Proposition \ref{prop:clfboundsforC225}]
It is enough to prove the case for slope stable vector bundle $F$ over $C_{2,2,5}$. We consider the Harder-Narasimhan filtration for $\iota_*F$ with respect to $\nu_{BN}$ as that in Lemma \ref{lem:HNforbnstab}:
$$0=F_0\subset F_1\subset \dots \subset F_m=\iota_*F.$$

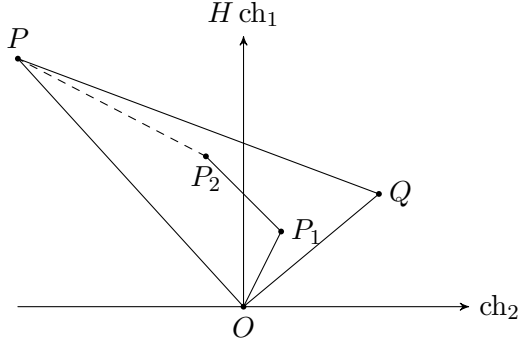
\begin{figure} [hbt!]
	\begin{centering}
		
		\begin{tikzpicture}[line cap=round,line join=round,x=1.0cm,y=1.0cm]

		\draw[->,color=black] (0,0) -- (0,3.6) node [above] {$H\ch_1$};
		\draw[->,color=black] (-3,0) -- (3,0);
		\draw[color=black] (0,0) -- (-3,3.3);
		\draw[color=black] (0,0) -- (1.8,1.5);
		\draw[color=black] (1.8,1.5) -- (-3,3.3);
		
		\draw[color=black] (0,0) -- (.5,1);
		\draw[dashed] (-.5,2) -- (-3,3.3);
		\draw[color=black] (-.5,2) -- (.5,1);
		
		\draw (0,0) node [below] {$O$};
		\draw (3,0) node [right] {$\ch_2$};
		\draw (1.8,1.5) node [right] {$Q$};
		\draw (-3,3.3) node [above] {$P$};
		\draw (.5,1) node [right] {$P_1$};
		\draw (-.5,2) node [below] {$P_2$};
		
		\begin{scriptsize}
		
		\fill [color=black] (0,0) circle (1.1pt);
		\fill [color=black] (-3,3.3) circle (1.1pt);
		\fill [color=black] (1.8,1.5) circle (1.1pt);
		\fill [color=black] (.5,1) circle (1.1pt);
		\fill [color=black] (-.5,2) circle (1.1pt);

		\end{scriptsize}
		
		\end{tikzpicture}
		
		\caption{The HN polygon for $\iota_*E$ is inside the triangle $OQP$.}
		
		\label{fig:HNpoly}
		
	\end{centering}
	
\end{figure}
We draw the points $P_i\coloneqq (\ch_2(F_i),H\ch_1(F_i))$, $1\leq i\leq m$ on the upper half plane $\H$. By Lemma \ref{lem:p2length} and the definition of the function $\clubsuit$, 
\begin{align}
h^0(F)=& \Hom(\cO_{S_{2,2}},\iota_*F)\leq \sum^{i=m}_{i=1}\Hom(\cO_{S_{2,2}},F_i/F_{i-1})\nonumber \\
\leq & \sum^{i=m}_{i=1}\rk(F_i/F_{i-1}) + \clubsuit(\ch_2(F_i/F_{i-1}),H\ch_1(F_i/F_{i-1}))\nonumber\\=&\sum^{i=m}_{i=1}\clubsuit(\overrightarrow{P_{i-1}P_i}).\label{eq:clubsum}
\end{align}
Let $P=P_n=((\mu-50)r,20r)$ and $Q=(x_q,y_q)$ be points on $\H$ such that $\frac{x_q}{y_q}$ is the upper bound for $\nu_{BN}^+(\iota_*F)$ and $\frac{x_p-x_q}{x_q-y_q}$ is the lower bound for $\nu_{BN}^-(\iota_*F)$ as that in Proposition \ref{prop:boundsforBNofiotaF}. The points $O,P_1,\dots,P,O$ then form the vertices of a convex polygon in the triangle  $OQP$ as that in Figure \ref{fig:HNpoly}. Now by Lemma \ref{lem:polygonmax}, we may estimate the upper bound for $h^0(F)$ by choosing suitable candidate point $P_1\coloneqq (x_1,y_1)$ in the triangle $OQP$.

We first treat with the case when $\mu\in (0,10]$, by Proposition \ref{prop:boundsforBNofiotaF}, the slope $\nu_{BN}^-(\iota_*F)\in (-\frac{5}2,-2)$.

\begin{itemize}
\item When $\mu\in(0,2)$, by Proposition \ref{prop:boundsforBNofiotaF}, the slope $\nu_{BN}^+(\iota_*F)\in (-\frac{5}2,-2).$ By (\ref{eq:clubsum}), 
$$h^0(F)/r\leq \frac{9}{82}\times 20+\frac{\mu-50}{41}=\frac{40}{41}+\frac{\mu}{41}.$$
\end{itemize}

 When $\mu\in[2,10]$, the point $Q$ is always with locus:
$$Q=((\frac{\mu}{10}-1)r,\frac{\mu}5r).$$
\begin{itemize}
\item When $\mu\in[2,\frac{5}2)$, by Proposition \ref{prop:boundsforBNofiotaF}, the slope $\nu_{BN}^+(\iota_*F)\in (-\frac{5}2,-\frac{3}2).$ By Lemma \ref{lem:polygonmax}, $P_1$ has two candidate positions: $P_1=Q$ or $P_1$ is on the line segment $PQ$ and $\frac{x_1}{y_1}=-2$.
\begin{itemize}
\item When $P_1=Q$, the equation (\ref{eq:clubsum}) is equal to 
$$\frac{7}{50}\frac{\mu}{5}r+\frac{1}{25}(\frac{\mu}{10}-1)r+\frac{9}{82}(20-\frac{\mu}{5})r+\frac{1}{41}(\frac{9\mu}{10}-49)r=(\frac{24}{25}+\frac{4\mu}{125})r.$$
\item In the second case, $P_1$ is at 
$$(x_1,y_1)=(\frac{-240\mu+4\mu^2+400}{90-5\mu}r,\frac{120\mu-2\mu^2-200}{90-5\mu}r).$$
Note that $y_1$ as a function of $\mu$ is convex down when $\mu\leq 10$. Substituting $y_1(2)=\frac{2}5r$ and $y_1(\frac{5}2)=\frac{35}{31}r$, we have $y_1\leq (\frac{226}{155}\mu-\frac{78}{31})r$.  The equation (\ref{eq:clubsum}) is equal to 
\begin{align}
& \frac{1}{16}y_1+\frac{9}{82}(20r-y_1)+\frac{1}{41}((\mu-50)r+2y_1) 
= (\frac{40}{41}+\frac{\mu}{41})r+\frac{1}{656}y_1 \label{eq:slope=-2est}\\ 
\leq & (\frac{40}{41}+\frac{\mu}{41}+\frac{1}{656}(\frac{226}{155}\mu-\frac{78}{31}))r 
=(\frac{33\mu}{1240}+\frac{241}{248})r.\nonumber
\end{align}
\end{itemize}
\item When $\mu\in[\frac{5}2,\frac{10}3)$, by Proposition \ref{prop:boundsforBNofiotaF}, the slope $\nu_{BN}^+(\iota_*F)\in (-\frac{5}2,-1).$ By Lemma \ref{lem:polygonmax}, $P_1$ has three candidate positions: $P_1=Q$ or $P_1$ is on the line segment $PQ$ with $\frac{x_1}{y_1}=-\frac{3}2$ or  $-2$.
\begin{itemize}
\item When $P_1=Q$, the equation (\ref{eq:clubsum}) is equal to 
$$\frac{5}{26}\frac{\mu}{5}r+\frac{1}{13}(\frac{\mu}{10}-1)r+\frac{9}{82}(20-\frac{\mu}{5})r+\frac{1}{41}(\frac{9\mu}{10}-49)r=(\frac{12}{13}+\frac{3\mu}{65})r.$$
\item In the second case, $P_1$ is at 
$$(x_1,y_1)=(\frac{-180\mu+3\mu^2+300}{190-6\mu}r,\frac{60\mu-\mu^2-100}{95-3\mu}r).$$
Note that $y_1$ as a function of $\mu$ is convex down when $\mu\leq 10$. Substituting $y_1(\frac{5}2)=\frac{1}2r$ and $y_1(\frac{10}3)=\frac{160}{153}r$, we have $y_1\leq (\frac{167}{255}\mu-\frac{58}{51})r$.  The equation (\ref{eq:clubsum}) is equal to 
\begin{align}
& \frac{1}{12}y_1+\frac{9}{82}(20r-y_1)+\frac{1}{41}((\mu-50)r+\frac{3}2y_1) = (\frac{40}{41}+\frac{\mu}{41})r+\frac{5}{492}y_1 \label{eq:slope=-32est} \\ 
\leq & (\frac{40}{41}+\frac{\mu}{41}+\frac{5}{492}(\frac{167}{255}\mu-\frac{58}{51}))r  =(\frac{295}{306}+ \frac{19\mu}{612})r.\nonumber
\end{align}
\item In the third case, the coordinate of $P_1$ is given in the second case of $\mu\in[2,\frac{5}2)$. The term $\frac{1}{656}y_1$ in (\ref{eq:slope=-2est}) is $\frac{35}{20336}r$ and $\frac{5}{1353}r$ when $\mu=\frac{5}2$ and $\frac{10}3$ respectively.  The term $\frac{5}{492}y_1$ in (\ref{eq:slope=-32est}) is $\frac{5}{984}r$ and $\frac{200}{18819}r$ when $\mu=\frac{5}2$ and $\frac{10}3$ respectively. Therefore, (\ref{eq:slope=-2est}) is always less than the estimation in the second case.
\end{itemize}
\item When $\mu\in[\frac{10}3,5)$, by Proposition \ref{prop:boundsforBNofiotaF}, $\nu_{BN}^+(\iota_*F)\in (-\frac{5}2,-\frac{1}2).$ By Lemma \ref{lem:polygonmax}, $P_1$ has four candidate positions: $P_1=Q$ or $P_1$ is on the line segment $PQ$ with $\frac{x_1}{y_1}=-1$ or $-\frac{3}2$ or  $-2$.
\begin{itemize}
\item When $P_1=Q$, the equation (\ref{eq:clubsum}) is equal to 
$$\frac{3}{10}\frac{\mu}{5}r+\frac{1}{5}(\frac{\mu}{10}-1)r+\frac{9}{82}(20-\frac{\mu}{5})r+\frac{1}{41}(\frac{9\mu}{10}-49)r=(\frac{4}{5}+\frac{2\mu}{25})r.$$
\item In the second case, $P_1$ is at 
$$(x_1,y_1)=(\frac{-120\mu+2\mu^2+200}{290-7\mu}r,\frac{120\mu-2\mu^2-200}{290-7\mu}r).$$
Note that $y_1$ as a function of $\mu$ is convex down when $\mu\leq 10$. Substituting $y_1(\frac{10}3)=\frac{2}3r$ and $y_1(5)=\frac{70}{51}r$, we have $y_1\leq (\frac{36}{85}\mu-\frac{38}{51})r$.  The equation (\ref{eq:clubsum}) is equal to 
\begin{align}
& \frac{1}{8}y_1+\frac{9}{82}(20r-y_1)+\frac{1}{41}((\mu-50)r+y_1) = (\frac{40}{41}+\frac{\mu}{41})r+\frac{13}{328}y_1 \nonumber \\ 
\leq & (\frac{40}{41}+\frac{\mu}{41}+\frac{13}{328}(\frac{36}{85}\mu-\frac{38}{51}))r  =(\frac{193}{204}+ \frac{7\mu}{170})r.\nonumber
\end{align}
\item The remaining cases can be eliminated by a similar calculation as that in the third case of when $\mu\in[\frac{5}2,\frac{10}3)$.
\end{itemize}

\item When $\mu\in[5,10]$, by Lemma \ref{lem:polygonmax}, $P_1$ has five candidate positions: $P_1=Q$ or $P_1$ is on the line segment $PQ$ with $\frac{x_1}{y_1}=-\frac{1}2$ or $-1$ or $-\frac{3}2$ or  $-2$.
\begin{itemize}
\item When $P_1=Q$, the equation (\ref{eq:clubsum}) is equal to 
$$\frac{\mu}{10}r+(\frac{\mu}{10}-1)r+\frac{9}{82}(20-\frac{\mu}{5})r+\frac{1}{41}(\frac{9\mu}{10}-49)r=\frac{\mu}{5}r.$$
\item In the second case, $P_1$ is at 
$$(x_1,y_1)=(\frac{-60\mu+\mu^2+100}{390-8\mu}r,\frac{60\mu-\mu^2-100}{195-4\mu}r).$$
Note that $y_1$ as a function of $\mu$ is convex down when $\mu\leq 10$. Substituting $y_1(5)=r$ and $y_1(10)=\frac{80}{31}r$, we have $y_1\leq (\frac{49}{155}\mu-\frac{18}{31})r$.  The equation (\ref{eq:clubsum}) is equal to 
\begin{align}
& \frac{1}{4}y_1+\frac{9}{82}(20r-y_1)+\frac{1}{41}((\mu-50)r+\frac{y_1}2) = (\frac{40}{41}+\frac{\mu}{41})r+\frac{25}{164}y_1 \nonumber \\ 
\leq & (\frac{40}{41}+\frac{\mu}{41}+\frac{25}{164}(\frac{49}{155}\mu-\frac{18}{31}))r  =(\frac{55}{62}+ \frac{9\mu}{124})r.\nonumber
\end{align}
\item The remaining cases can be eliminated by a similar calculation as that in the third case of when $\mu\in[\frac{5}2,\frac{10}3)$.
\end{itemize}
\end{itemize}

We then treat with the case when $\mu\in [30,40]$.
\begin{itemize}
\item When $\mu\geq 10+20\sqrt 2$, as $\nu_{BN}^-(\iota_*F)=-\frac{3}2$, by Lemma \ref{lem:polygonmax}, we may assume $P_1=Q$ to compute (\ref{eq:clubsum}). The coordinate of $Q$ is 
\begin{align*}
(x_q,y_q)=\begin{cases}
 ((\mu-32)r,8r),& \text{ when } \mu\in[39,40];\\
 (\frac{23\mu-610}{41}r,\frac{12\mu-140}{41}r), & \text{ when } \mu\in[10+20\sqrt 2,39].
\end{cases}
\end{align*}
 The equation (\ref{eq:clubsum}) is equal to 
\begin{align}
& \frac{1}2y_q+x_q+\frac{1}{12}(20r-y_q) =\begin{cases}
 (\mu-27 )r,& \text{ when } \mu\in[39,40];\\
 \frac{28\mu-600}{41}r, & \text{ when } \mu\in[10+20\sqrt 2,39].
\end{cases}\nonumber
\end{align}
\item When $\mu\leq 10+20\sqrt 2$, the point $Q$ is always with locus:
$$((\frac{3\mu}{10}-5)r,\frac{\mu}5r).$$
We may consider when $P_1=Q$ or $P_1$ is on the line segment $OQ$ such that $\frac{x_p-x_1}{y_p-y_1}=-\frac{3}2$.
\begin{itemize}
\item The second case is the same computation as that in the second case of when $\mu\geq 10+20\sqrt 2$.
\item When $P_1=Q$, the equation (\ref{eq:clubsum}) is equal to 
\begin{align}
& \frac{1}2\frac{\mu}5r+(\frac{3\mu}{10}-5)r+\frac{7}{50}(20-\frac{\mu}5)r +\frac{1}{25}(\frac{7\mu}{10}-45)r =(\frac{2\mu}5-4)r.\nonumber
\end{align}
\end{itemize}
\end{itemize}

When $\mu= 37$, $\frac{2\mu}5-4=10.8=\frac{11\mu}{15}-\frac{49}3>10\frac{26}{41}=\frac{28\mu-600}{41}>\mu-27$.

When $\mu=40$, $\frac{2\mu}5-4=12<\frac{28\mu-600}{41}=12\frac{28}{41}<13=\frac{11\mu}{15}-\frac{49}3=\mu-27$. 

Note the slope of $\mu$ in the bound in each case, the bound in Proposition \ref{prop:clfboundsforC225} holds. 
\end{proof}

\section{Bogomolov-Gieseker type inequality for surfaces $S_{2,5}$ and quintic threefolds}
\label{sec:5}

The goal of this section is to prove the stronger Bogomolov-Gieseker type inequality for the second Chern character of slope stable sheaves on a quintic threefold. Our strategy is to first reduce this to the same inequality for a surface on the quintic threefold.

The following Feyzbakhsh's restriction lemma \cite{Fe16} will be one of the key tools to reduce Bogomolov-Gieseker type inequality for higher dimensional varieties to surfaces.

\begin{Lem}
Let $(X,H)$ be a polarized smooth projective variety with dimension $n=2$ or $3$. Let $E$ be a coherent sheaf in $\Coh^{0,H}(X)$. Suppose there exists $\alpha>0$ and $m\in \Z_{>0}$ such that 
\begin{itemize}
\item $E(-mH)[1]$ is in $\Coh^{0,H}(X)$;
\item both $E$ and $E(-mH)[1]$ are $\nu_{\alpha,0,H}$-tilt stable;
\item $\nu_{\alpha,0,H}(E)=\nu_{\alpha,0,H}(E(-mH)[1])$.
\end{itemize}
Then for a generic smooth irreducible subvariety $Y\in|mH|$, the restricted sheaf $E|_Y$ is $\mu_{H_Y}$-slope semistable on $Y$. Moreover, $\rk(E)=\rk(E|_Y), H_Y^{n-2}\ch_1(E|_Y)=mH^{n-1}\ch_1(E) $ and when $n=3$, $\ch_2(E|_Y)=mH\ch_2(E)$.
\label{lem:soheylares}
\end{Lem}
\begin{proof}
Note that $E(-mH)[1]$ is $\nu_{\alpha,0,H}$-tilt stable, for any torsion sheaf $T$ supported  on a variety with codimension not less than $2$, we have $\Hom(T,E(-mH)[1])=0$. In particular, $E$ is a reflexive sheaf, the singular locus of $E$ is of codimension at least $3$. For any smooth irreducible $Y\in|mH|$ avoiding the singular locus, the restricted sheaf $E|_Y$ is locally free on $Y$. In addition, $\rk(E)=\rk(E|_Y), H_Y^{n-2}\ch_1(E|_Y)=mH^{n-1}\ch_1(E) $. 

Suppose $E|_Y$ is not semistable, then there is a destabilizing subobject $F\hookrightarrow E|_Y$ in $\Coh(Y)$ such that $F$ is locally free and $\mu_{H_Y}(E|_Y)<\mu_{H_Y}(F)$. Denote the embedding by $\iota:Y\hookrightarrow X$. Then 
\begin{align*}
\nu_{\alpha,0,H}(\iota_*(E|_Y)) =&\frac{H^{n-2}\ch_2(\iota_*(E|_Y))}{H^{n-1}\ch_1(\iota_*(E|_Y))}=\frac{H_Y^{n-2}\ch_1(E|_Y)-\frac{1}2mH^{n-1}_{Y}\rk(E|_Y)}{H^{n-1}_Y\rk(E|_Y)}\\
 =& \mu_{H_Y}(E|_Y)-\frac{m}2<\mu_{H_Y}(F)-\frac{m}2 =\nu_{\alpha,0,H}(\iota_*F).
\end{align*}  
Therefore $\iota_*(E|_Y)$ is not $\nu_{\alpha,0}$-tilt semistable. However, the object $\iota_*(E|_Y)$ is the extension of $E$ and $E(-mH)[1]$ in $\Coh^{0,H}(X)$. Since both $E$ and $E(-mH)[1]$ are $\nu_{\alpha,0,H}$-tilt stable with the same slope in $\Coh^{0,H}(X)$, any of their extension is  $\nu_{\alpha,0,H}$-tilt semistable. We get the contradiction, and $E|_Y$ must be $\mu_{H_Y}$-slope semistable.
\end{proof}
Let $S_{2,5}\subset \mathbf{P}^4$ be a smooth irreducible projective surface which is the complete intersection of a quadratic hypersurface and a quintic hypersurface. Denote $H=[\mathcal O_{S_{2,5}}(1)]$. By the Clifford type inequality for $C_{2,2,5}\in\abs{2H}$ in Proposition \ref{prop:clfboundsforC225}, we have the following stronger Bogomolov-Gieseker type inequality for stable objects in $D^b(S_{2,5})$.

\begin{Prop}
Let $F$ be an object in $D^b(S_{2,5})$ such that $\frac{H\ch_1(F)}{H^2\rk(F)}\in(0,1)$. Suppose $F$ is $\nu_{\alpha,0,H}$-tilt stable or $\nu_{\alpha',1,H}$-tilt stable for some $\alpha>0$ or $\alpha'>\frac{1}2$, then 
\begin{align}\label{eq:41}
\frac{\ch_2(F)}{H^2\rk(F)}\leq \begin{cases}
\frac{3}2\left(\frac{H\ch_1(F)}{H^2\rk(F)}\right)^2-\frac{H\ch_1(F)}{H^2\rk(F)}, & \text{when } \frac{H\ch_1(F)}{H^2\rk(F)} \in (0, \frac{1}{10}]\cup [\frac{9}{10},1);\\
-\frac{4}{15}\frac{H\ch_1(F)}{H^2\rk(F)}-\frac{7}{120}, & \text{when } 
\frac{H^2\ch_1(F)}{H^3\rk(F)} \in [\frac{1}{10}, \frac{1}4];\\
\frac{1}2\frac{H\ch_1(F)}{H^2\rk(F)} - \frac{1}4 , & \text{when } \frac{H\ch_1(F)}{H^2\rk(F)} \in [\frac{1}4, \frac{3}4];\\
\frac{19}{15}\frac{H\ch_1(F)}{H^2\rk(F)}-\frac{33}{40}, & \text{when } \frac{H\ch_1(F)}{H^2\rk(F)} \in [\frac{3}4, \frac{9}{10}].
\end{cases}
\end{align}
\label{prop:bgfors25}
\end{Prop}

\begin{proof}
Suppose there is some $\nu_{\alpha,0,H}$ or $\nu_{\alpha',1,H}$-tilt stable object $F$ with $\frac{H\ch_1(F)}{H^2\rk(F)}\in(0,1)$ violating the inequality (\ref{eq:41}), we may assume that $F$ is with the minimum discriminant $\db_H$ among such objects. Suppose $F$ becomes strictly $\nu_{\alpha,0,H}$-tilt (or $\nu_{\alpha',1,H}$-tilt) semistable for some $\alpha>0$ (or $\alpha'>\frac{1}2$), then as the shape of the curve (\ref{eq:41}) is convex (see Figure \ref{fig:strbgcurve}), there exists a Jordan-H\"older  factor $F_i$ with $\frac{H\ch_1(F_i)}{H^2\rk(F_i)}\in(0,1)$  which also violates the inequality (\ref{eq:41}). By Lemma \ref{lem:discriminnatdec}, this violates the minimum assumption on $\db_H(F)$. 

If $F$ becomes strictly $\nu_{\alpha,\beta_0,H}$-tilt semistable at the vertical wall for $\beta_0=\frac{H\ch_1(F)}{H^2\rk(F)}$ and some $\alpha>\frac{\beta
_0^2}2$, we may  assume that $F\in\Coh^{\beta_0,H}(S_{2,5})$,  then each torsion Jordan-H\"older factor of $F$ has $\ch_2\geq 0$. Since for any other Jordan-H\"older factor $F_j$ we have $0>\rk(F_j)\geq \rk(F)$, there exists a factor $F_i$ with $\frac{H\ch_1(F_i)}{H^2\rk(F_i)}=\frac{H\ch_1(F)}{H^2\rk(F)}$ and $\frac{\ch_2(F_i)}{H^2\rk(F_i)}\geq \frac{\ch_2(F)}{H^2\rk(F)}$. In particular, the object $F_i$ also violates the inequality (\ref{eq:41}) and $\db_H(F_i)\leq \db_H(F)$. By Lemma \ref{lem:nestedandopen}, the object $F_i$ is  $\nu_{\alpha,0,H}$-tilt stable and $\nu_{\alpha,1,H}$-tilt stable for $\alpha\gg 0$. By the minimum assumption on $\db_H(F)$, the equality holds, we may just choose $F$ to be  $F_i$. By the previous argument, the object  $F$ is $\nu_{\alpha,0,H}$-tilt stable for all $\alpha>0$ and $\nu_{\alpha,1,H}$-tilt stable for all $\alpha>\frac{1}2$.  
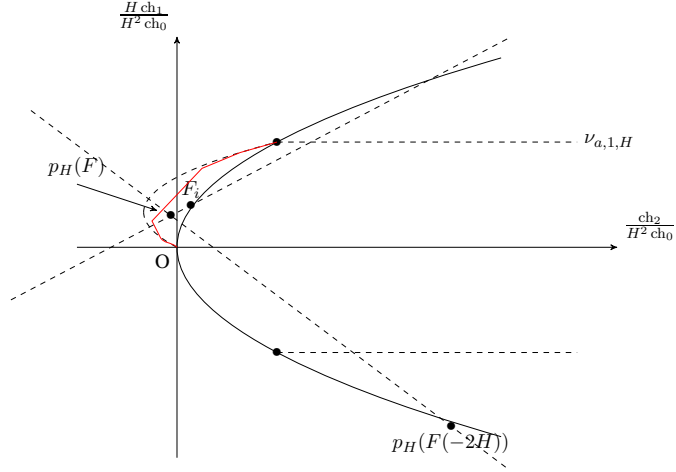
\begin{figure}[hbt!]
\begin{center}
\scalebox{0.7}{

\begin{tikzpicture}[scale=2]
\tikzset{%
    add/.style args={#1 and #2}{
        to path={%
 ($(\tikztostart)!-#1!(\tikztotarget)$)--($(\tikztotarget)!-#2!(\tikztostart)$)%
  \tikztonodes},add/.default={.2 and .2}}
}

\newcommand\XA{1.9}

\draw[->] (-0.5*\XA,0) -- (2.2*\XA,0) node[above right] {$\frac{\ch_2}{H^2\ch_0}$};

\draw[->] (0,-2)-- (0,0) node [below left] {O} --  (0,2) node[above left] {$\frac{H\ch_1}{H^2\ch_0}$};

\draw[rotate=-90] (0,0) parabola (1.8,1.62*\XA);
\draw[rotate=-90] (0,0) parabola (-1.8,1.62*\XA);
\draw[rotate=-90,dashed] (-1/3,-1/6*\XA) parabola (0,0);
\draw[rotate=-90,dashed] (-1/3,-1/6*\XA) parabola (-1,1/2*\XA);

\coordinate (O) at (0,0);
\coordinate(O1) at (0.5*\XA,1);
\coordinate(O-1) at (0.5*\XA,-1);
\coordinate (AB) at (-0.03*\XA,0.3);
\coordinate(E-2) at(1.37*\XA,-1.7);
\coordinate(E1) at(0.07*\XA,.4);
\node at (AB) {$\bullet$};
\node at (O1) {$\bullet$};
\node at (O-1) {$\bullet$};
\node at (E-2) {$\bullet$};
\node[below] at (E-2) {$p_H(F(-2H))$};
\node at (E1) {$\bullet$};
\node[above] at (E1) {$F_i$};

\draw[red](0,0) to[out=135,in=-65] (-17/200*\XA,1/10)--(-1/8*\XA,1/4)--(1/8*\XA,3/4)--(63/200*\XA,9/10)--(1/2*\XA,1);
\draw node[above] at (-0.5*\XA,0.6) {$p_H(F)$};
\draw[->] (-0.5*\XA,0.6)--(-0.1*\XA,0.35);
\draw[dashed] (O-1) -- (2*\XA,-1);
\draw[dashed] (O1) -- (2*\XA,1) node[right] {$\nu_{a,1,H}$};

\draw [add= .5 and 0.2,dashed] (AB) to (E-2);
\draw [add= 8 and 16,dashed] (AB) to (E1);
\end{tikzpicture}
}
\end{center}
\caption{The line through $p_H(F)$ and $p_H(F(-2H)$.} \label{fig:strbgcurve}
\end{figure}

We may assume that $F\in \Coh^{0,H}(S_{2,5})$ and $F[1]\in \Coh^{1,H}(S_{2,5})$. By the inequality (\ref{eq:41}), we always have $$\frac{\ch_2(F)}{H^2\rk(F)}> 
\frac{3}2\left(\frac{H\ch_1(F)}{H^2\rk(F)}\right)^2-\frac{H\ch_1(F)}{H^2\rk(F)}.$$
The line through $p_H(F)=:(a,b)$ and $p_H(F(-2H)[1])=(a-2b+2,b-2)$
has equation $$(b - 1) Y - X = -a + b^2 - b.$$ Note that $a>\frac{3}2b^2-b$, the  line will intersect $(\alpha_0,0)$ for some  $\alpha_0>0$ and $(\alpha'_0,-1)$ for some $\alpha_0'>\frac{1}2$. By previous discussions, the object $F$ is a coherent sheaf and $\nu_{\alpha_0,0,H}$-tilt stable in $\Coh^{0,H}(S_{2,5})$. The object $F(-2H)[1]$ is $\nu_{\alpha'_0,-1,H}$-tilt stable in $\Coh^{-1,H}(S_{2,5})$ and therefore also $\nu_{\alpha_0,0,H}$-tilt stable in $\Coh^{0,H}(S_{2,5})$ by Lemma \ref{lem:nestedandopen}. Since $p_H(F)$, $p_H(F(-2H)[1])$ and $(\alpha_0,0)$ are collinear, by Lemma \ref{lem:nestedandopen} ($b'$), we have $$\nu_{\alpha_0,0,H}(F)=\nu_{\alpha_0,0,H}(F(-2H)[1]).$$
By Lemma \ref{lem:soheylares}, let $C_{2,2,5}\in |2H|$ be a smooth irreducible curve, then $F|_{C_{2,2,5}}$ is semistable with 
$$\rk(F|_{C_{2,2,5}})=\rk(F) \text{ and } \deg(F|_{C_{2,2,5}})=2H\ch_1(F).$$

Without loss of generality, we may assume $\frac{H\ch_1(F)}{H^2\rk(F)}\leq \frac{1}2$, as otherwise we may use $F^\vee(H)$ instead. Note that $\Hom(\mathcal O_{S_{2,5}},F(-2H))=0$ and $\Hom(\mathcal O_{S_{2,5}},F^\vee)=0$, by Hirzebruch-Riemann-Roch, we have
\begin{align}
& \ch_2(F)-H\ch_1(F)+15\rk(F)=\chi(\mathcal O_{S_{2,5}},F) \label{eq:43}\\
\leq & \hom(\mathcal O_{S_{2,5}},F)+\hom(\mathcal O_{S_{2,5}},F[2])=\hom(\mathcal O_{S_{2,5}},F)+\hom(\mathcal O_{S_{2,5}},F^\vee(2H))\\
\leq & \hom(\mathcal O_{C_{2,2,5}},F|_{C_{2,2,5}})+\hom(\mathcal O_{C_{2,2,5}},F^\vee(2H)|_{C_{2,2,5}}). \label{eq:42}
\end{align}
We now apply Proposition \ref{prop:clfboundsforC225} and discuss three different cases on the slope of $F|_{C_{2,2,5}}$. Note that the slope $\mu(F|_{C_{2,2,5}})=20\frac{H\ch_1(F)}{H^2\rk(F)}\in (0,10]$.
\begin{itemize}
\item When $\mu\in(0,2)$,  by Proposition \ref{prop:clfboundsforC225}, the equation (\ref{eq:42}) 
\begin{align*}
\leq & \left(\frac{40}{41}+\frac{\mu}{41}+\frac{11}{15}(40-\mu)-\frac{49}3\right)\rk(F) <(14-\frac{2}3\mu)\rk(F) = 14 \rk(F) -\frac{4}3H\ch_1(F).
\end{align*}
This is less than that in equation (\ref{eq:43}), since in this case we have assumed that 
$$\frac{\ch_2(F)}{H^2\rk(F)}>\frac{3}2\left(\frac{H\ch_1(F)}{H^2\rk(F)}\right)^2-\frac{H\ch_1(F)}{H^2\rk(F)}>-\frac{1}{10}\implies \ch_2(F)+\rk(F)>0.$$
\item When $\mu\in[2,5)$, note that in  Proposition \ref{prop:clfboundsforC225}, the right hand side in equation (\ref{eq:clf}) is always less than or equal to
$
\begin{cases}
\frac{7\mu}{120}+\frac{109}{120}, & \text{ when } \mu\in[2,5);\\
\frac{23\mu}{45}-\frac{71}{9}, & \text{ when } \mu\in(35,38].
\end{cases}
$

Therefore, the equation (\ref{eq:42}) 
\begin{align*}
\leq & \left(\frac{109}{120}+\frac{7\mu}{120}+\frac{23}{45}(40-\mu)-\frac{71}9\right)\rk(F) =13\frac{167}{360}\rk(F)-\frac{163}{180}H\ch_1(F).
\end{align*}
By (\ref{eq:43}-\ref{eq:42}) and the assumption on $F$ that it violates (\ref{eq:41}), we have 
$$-\frac{4}{15}H\ch_1(F)-\frac{7}{12}\rk(F)<\ch_2(F)\leq \frac{17}{180}H\ch_1(F)-1\frac{193}{360}\rk(F).$$

This is not possible since $H\ch_1(F)<2.5\rk(F)$.
\item When $\mu\in[5,10]$,  in Proposition \ref{prop:clfboundsforC225}, the equation (\ref{eq:clf}) is always less than or equal to $\frac{3\mu}{20}+\frac{1}2$. Therefore, the equation (\ref{eq:42}) 
\begin{align*}
\leq & \left(\frac{1}{2}+\frac{3\mu}{20}+\frac{2}{5}(40-\mu)-4\right)\rk(F) =  12\frac{1}2 \rk(F) -\frac{1}2H\ch_1(F).
\end{align*}

By (\ref{eq:43}-\ref{eq:42}), the object $F$ satisfies equation (\ref{eq:41}).
\end{itemize}
In either case of $\mu$, we always get contradiction. Therefore, any $\nu_{\alpha,0,H}$ or $\nu_{\alpha',1,H}$-tilt stable object $F$ with $\frac{H\ch_1(F)}{H^2\rk(F)}\in(0,1)$ satisfies the inequality (\ref{eq:41}).
\end{proof}

\begin{Cor}
Let $F$ be a torsion free $\mu_H$-slope semistable sheaf on $S_{2,5}$, then the numerical Chern characters of $F$ satisfy equation (\ref{eq:41}).
\label{cor:bgfors25}
\end{Cor}
\begin{proof}
This is by Proposition \ref{prop:bgfors25} and by noticing that $F$ is $\nu_{\alpha,0,H}$-tilt stable for $\alpha\gg 0$. 
\end{proof}

\begin{Cor}
Let $(X,H)$ be a smooth projective quintic threefold, $F$ be an object in $D^b(X)$ such that $\frac{H^2\ch_1(F)}{H^3\rk(F)}\in(0,1)$. Suppose $F$ is $\nu_{\alpha,0,H}$-tilt stable or $\nu_{\alpha',1,H}$-tilt stable for some $\alpha>0$ or $\alpha'>\frac{1}2$, then (\ref{eq:41}) holds for $F$ if one replaces $\ch_2(F)$, $H\ch_1(F)$ and $H^2\rk(F)$ by $H\ch_2(F)$, $H^2\ch_1(F)$ and $H^3\rk(F)$ respectively.
\label{cor:bgforX5}
\end{Cor}
\begin{proof}
Suppose there is some $\nu_{\alpha,0,H}$ or $\nu_{\alpha',1,H}$-tilt stable object $F$ with $\frac{H^2\ch_1(F)}{H^3\rk(F)}\in(0,1)$ violating the inequality (\ref{eq:41}), we may assume that $F$ is with the minimum discriminant $\db_H$ among such object. By the same argument as that in Proposition \ref{prop:bgfors25}, we may assume that $F\in \Coh^{0,H}(S_{2,5})$ is $\nu_{\alpha,0,H}$-tilt stable for all $\alpha>0$ and $\nu_{\alpha',1,H}$-tilt stable for all $\alpha'>\frac{1}2$. 

Due to the same argument as that in Proposition \ref{prop:bgfors25} and Lemma \ref{lem:soheylares}, there exists $S_{2,5}\in |2H|$ such that $F|_{S_{2,5}}$ is $\mu_{H_{S_{2,5}}}$slope semistable with 
$$H_{S_{2,5}}^2\rk(F|_{S_{2,5}})=2H^3\rk(F), H_{S_{2,5}}\ch_1(F|_{S_{2,5}})=2H\ch_1(F)  \text{ and } \ch_2(F|_{S_{2,5}})=2H\ch_2(F).$$
Note that the characters of $F|_{S_{2,5}}$ violate the inequality (\ref{eq:41}), by Corollary \ref{cor:bgfors25}, we get the contradiction.
\end{proof}

We restate Corollary \ref{cor:bgforX5} as a theorem in the following neater version. The inequality is slightly weaker but can be applied more effectively in the proof for our main theorem on the third Chern character. It can also be viewed as a stronger Bogomolov-Gieseker type inequality in the classical sense.   

\begin{Thm}
 Let $(X,H)$ be a smooth projective quintic threefold, $F$ be a slope semistable sheaf in $\Coh(X)$ (or  a $\nu_{\alpha,0,H}$-tilt semistable object for some $\alpha>0$, especially Brill-Noether semistable object in $\Coh^{0,H}(X)$). Suppose $\frac{H^2\ch_1(F)}{H^3\rk(F)}\in[-1,1]$, then
\begin{align*}
\frac{H\ch_2(F)}{H^3\rk(F)}\leq \begin{cases}
-\frac{1}2\abs{\frac{H^2\ch_1(F)}{H^3\rk(F)}}, & \text{when } \abs{\frac{H^2\ch_1(F)}{H^3\rk(F)}} \in [0, \frac{1}4];\\
\frac{1}2\abs{\frac{H^2\ch_1(F)}{H^3\rk(F)}} - \frac{1}4 , & \text{when } \abs{\frac{H^2\ch_1(F)}{H^3\rk(F)}} \in [\frac{1}4, \frac{3}4];\\
\frac{3}2\abs{\frac{H^2\ch_1(F)}{H^3\rk(F)}} - 1 , & \text{when } \abs{\frac{H^2\ch_1(F)}{H^3\rk(F)}} \in [\frac{3}4, 1].
\end{cases}
\end{align*}
The `$=$' can hold only when $\frac{H^2\ch_1(F)}{H^3\rk(F)}\in\frac{1}4\Z$. Moreover, when $\abs{\frac{H^2\ch_1(F)}{H^3\rk(F)}} \in [0, \frac{1}{10}]\cup[\frac{9}{10},1]$, we have $\frac{H\ch_2(F)}{H^3\rk(F)}\leq\frac{3}2\abs{\frac{H^2\ch_1(F)}{H^3\rk(F)}}^2-\abs{\frac{H^2\ch_1(F)}{H^3\rk(F)}}$.
\label{thm:BGforslopestableX5}
\end{Thm}
\begin{Rem}[Other projective Calabi-Yau threefolds] One may expect to generalize the argument to some other  Calabi-Yau threefolds that can be realized as a complete intersection in $\mathbf P^N$ for $N$ not too large. To do that one could replace  $S_{2,5}$ by a smooth subvariety $Y\in|2H|$ and the curve $C_{2,2,5}$ by $C\in |2H_Y|$. Evidently, the inequality in  Proposition \ref{prop:bgfors25} does not hold for $Y$ in general. The first non-trivial task is to find a suitable Bogomolov-Gieseker type inequality for $Y$, the inequality needs to be sharp for some value of $\frac{H\ch_1}{\rk}$, especially when $\frac{H\ch_1}{\rk}=\frac{H^3}{\chi(\mathcal O_X(H))-1}$, so that it is strong enough to prove Proposition \ref{prop:BGforbnstab}. The next task is to estimate a Clifford type inequality for the curve $C$. It is worth to mention that some results in \cite{Lange2015g4,Lange2015g5,Lange2017g6} may help. Also one may consider to use the method in Section \ref{sec:clf} by finding a suitable surface containing the curve. As a summary, our methods are expect to be generalized to some other Calabi-Yau threefolds, meanwhile it seems that each deformation type will require much computation.
\label{rem:othercy3}
\end{Rem}

\bibliography{all}                      
\bibliographystyle{halpha}     

\end{document}